\documentclass{irmaart}
\usepackage{graphicx,epsfig,color}
\usepackage{dashbox}
\usepackage[all]{xy}
\usepackage{hyperref}

\theoremstyle{plain}
\newtheorem{theorem}{Theorem}[section]
\newtheorem{corollary}[theorem]{Corollary}
\newtheorem{lemma}[theorem]{Lemma}
\newtheorem{proposition}[theorem]{Proposition}

\newtheorem{example}[theorem]{Example}

\newtheorem*{coro}{Corollary} %%%% for unnumbered statements

\newtheorem{definition}[theorem]{Definition}
\newtheorem{remark}[theorem]{Remark}

\newcommand{\Hom}{\operatorname{Hom}}
\newcommand{\Ker}{\operatorname{Ker}}

\newcommand{\cL}{\mathcal{L}}

\newcommand{\C}{{\mathbb{C}}}

\newcommand{\F}{{\mathbb{F}}}     
\renewcommand{\H}{{\mathbb{H}}}

\newcommand{\N}{{\mathbb{N}}}     
\newcommand{\Q}{{\mathbb{Q}}}     
\newcommand{\R}{{\mathbb{R}}}

\newcommand{\Z}{{\mathbb{Z}}}     
\newcommand{\kk}{{\mathbf{k}}}     
\newcommand{\p}{{\partial}}     

\newcommand{\one}     
{{{\mathchoice \mathrm{ 1\mskip-4mu l} \mathrm{ 1\mskip-4mu l}     
\mathrm{ 1\mskip-4.5mu l} \mathrm{ 1\mskip-5mu l}}}}     
\newcommand{\Id}{\operatorname{Id}}     
\newcommand{\eps}{{\varepsilon}}     
\newcommand{\lcan}{\lambda_{\mathrm{can}}}
\newcommand{\ocan}{\omega_{\mathrm{can}}}
\newcommand{\Hlin}{H^{\mbox{\tiny{lin}}}} 
\newcommand{\Hcont}{H^{\mbox{\tiny{cont}}}}
\newcommand{\HrSFT}{H^{\mbox{\tiny{rSFT}}}}
\newcommand{\HSFT}{H^{\mbox{\tiny{SFT}}}}

\begin{document}

\title{Morse theory, closed geodesics, and the homology of free loop spaces}

\author{Alexandru Oancea \\ \ \\ \large{Appendix by Umberto Hryniewicz}}
\address{A.O.: Sorbonne Universit\'es, UPMC Univ Paris 06,\\
UMR 7586, Institut de Math\'ematiques de Jussieu-Paris Rive Gauche, \\
Case 247, 4 place Jussieu, F-75005, Paris, France.\\
Email: \tt{alexandru.oancea@imj-prg.fr}\\
\qquad \\
U.H.: Universidade Federal do Rio de Janeiro,\\
Instituto de Matem\'atica, Cidade Universit\'aria,\\
CEP 21941-909, Rio de Janeiro, Brazil.\\
Email: \tt{umberto@labma.ufrj.br}
}

\maketitle

\begin{abstract} We give a survey of the existence problem for closed geodesics. The free loop space plays a central role, since closed geodesics are critical points of the energy functional. As such, they can be analyzed through variational methods, and in particular Morse theory. The topics that we discuss include: Riemannian background, the Lyusternik-Fet theorem, the Lyusternik-Schnirelmann principle of subordinated classes, the Gromoll-Meyer theorem, Bott's iteration of the index formulas, homological computations using Morse theory, $SO(2)$- vs. $O(2)$-symmetries, Katok's examples and Finsler metrics, relations to symplectic geometry, and a guide to the literature. 

The Appendix by Umberto Hryniewicz gives an account of the proof of the existence of infinitely many closed geodesics on the $2$-sphere.
\end{abstract}



\section{Introduction}

The study of geodesics on Riemannian manifolds was historically one of the driving forces in the development of the calculus of variations. The goal of the present paper is to present an overview of results related to the following two questions. 

\smallskip
\noindent {\bf Questions.} {\it Does every closed Riemannian manifold $M$ carry a closed geodesic? If yes, how many of them?} 
\smallskip

\noindent Here \emph{closed manifold} means a manifold that is compact and has no boundary. 

The first question admits a relatively easy answer if the manifold is not simply connected: 
free homotopy classes of loops on $M$ are in one-to-one bijective correspondence with conjugacy classes in $\pi_1(M)$. One can minimize length (or, equivalently, energy) within such a nontrivial free homotopy class, and one of the first successes of the calculus of variations was to establish rigorously that such a minimizing procedure is effective and produces a closed geodesic. The situation is subtler if the manifold is simply connected, and the question was answered in the affirmative by Lyusternik and Fet in their celebrated 1951 paper~\cite{Lyusternik-Fet}. We explain their theorem in~\S\ref{sec:Lyusternik-Fet}.

In order to phrase the second question in a more satisfactory way, let us call a non-constant closed geodesic \emph{prime} if it is not the iterate of some other closed geodesic, and call two geodesics \emph{geometrically distinct} if they differ as subsets of $M$. The actual expectation (which we phrase as a question) is the following. 

\smallskip
\noindent {\bf Question.} {\it Does every closed Riemannian manifold $M$ carry infinitely many geometrically distinct prime non-constant closed geodesics?}
\smallskip

The answer is not known in full generality. We refer to~\S\ref{sec:comments} for a description of the current state of the art. In the spirit of the discussion of the first question, one can cook up a class of non-simply connected manifolds for which the answer is easy, namely manifolds whose fundamental group has infinitely many conjugacy classes, which are not iterates of a finite set of conjugacy classes. For such manifolds the minimizing procedure described above yields infinitely many closed geodesics, one for each free homotopy class, among which there are necessarily infinitely many geometrically distinct ones. At the other end of the hierarchy, if the finite group is finite the question reduces to the simply connected case. The core of the matter is thus again the simply connected case, and the   breakthrough in this direction was achieved by Gromoll and Meyer~\cite{Gromoll-Meyer-JDG1969}. We explain their theorem in~\S\ref{sec:GM}.

Note that, in general, a prime closed geodesic is not \emph{simple}, i.e. it does have self-intersections~: one can explicitly determine/bound the number of simple closed geodesics in some particular cases -- a $2$-dimensional ellipsoid with unequal axes of approximately equal length has exactly three simple closed geodesics~\cite{Morse}. We refer to~\cite{Hingston1990} for a beautiful account of the problem of the existence of closed geodesics on $S^2$, see also the discussion in Appendix~\ref{app:S2} written by Umberto Hryniewicz; we now know that there are always infinitely many geometrically distinct prime geodesics on $S^2$ thanks to work of Bangert, Franks, Hingston and Angenent~\cite{Bangert1993,Franks1992,Hingston1993,Angenent}.

Closed geodesics $\gamma:S^1=\R/\Z\to M$ are critical points of the energy functional 
$$
E(\gamma)=\int_{S^1}|\dot\gamma(t)|^2\, dt
$$
defined on the space $\Lambda M$ of free loops $\gamma:S^1\to M$ (one convenient setup is to use loops of Sobolev class $H^1$). The study of this smooth functional was the main motivation behind the invention by Marston Morse of ``Morse theory''~\cite{Morse, Milnor-Morse_theory}. This establishes a close relationship between the critical points of $E$ and the topology of the Hilbert manifold $\Lambda M$. 

The paper is organized as follows. In~\S\ref{sec:energy} we recall basic facts of Riemannian geometry and in~\S\ref{sec:Morse} we provide an account of Morse theory, with emphasis on the energy functional. In~\S\ref{sec:Lyusternik-Fet} we give the proof of the famous Lyusternik-Fet theorem, and explain the principle of subordinated classes of Lyusternik-Schnirelmann, which allows to detect distinct critical levels. In~\S\ref{sec:GM} we give an outline of the proof of the celebrated Gromoll-Meyer theorem, and explain Bott's iteration formulae for the index of closed geodesics. In~\S\ref{sec:equivariant} we give an overview of results due to Klingenberg, Takens, Hingston and Rademacher related to the problem of the existence of infinitely many closed geodesics and which go beyond the Gromoll-Meyer theorem. We also motivate the use of equivariant homology in the study of the closed geodesics problem. In~\S\ref{sec:completing} we explain how to compute the homology of the space of free loops on spheres and projective spaces using Morse theory. In~\S\ref{sec:symplectic} we discuss the relationship between the existence problem for  closed geodesics and Hamiltonian dynamics, with an emphasis on some remarkable examples of Finsler metrics due to Katok. 
The paper ends with a quick guide to the literature and with an Appendix by Umberto Hryniewicz on geodesics on the $2$-sphere. 

We would like to draw from the start the reader's attention to the classical survey paper by Bott~\cite{Bott-oldnew}.

\noindent {\bf Notation.} 

$\cL M$ denotes the space of smooth free loops on a manifold $M$

$\Lambda M$  denotes the space of free loops of Sobolev class $H^1=W^{1,2}$ on $M$

$\mathbb{F}_p$ denotes the field with $p$ elements, for $p\ge 2$ prime

$b_k(\Lambda M;\mathbb{F}_p):=\mathrm{rk}\, H_k(\Lambda M; \mathbb{F}_p)$ is the $k$-th Betti number of $\Lambda M$, considered with $\mathbb{F}_p$-coefficients; 
this is the same as the $k$-th Betti number of $\cL M$ with $\mathbb{F}_p$-coefficients.

\noindent {\bf Acknowledgements.} The author is particularly grateful for help, suggestions, and/or inspiration to Nancy Hingston, Umberto Hryniewicz, Janko Latschev, and to the anonymous referee. Particular thanks go to Umberto Hryniewicz for having contributed the Appendix. The author is partially supported by the ERC Starting Grant 259118-STEIN.

\section{The energy functional} \label{sec:energy}

\noindent {1. \it Generalities on Riemannian manifolds~\cite{Chavel-Riemannian_Geometry,GHL}.} Given a manifold $M$, denote $\mathcal{X}(M):=\Gamma(TM)$ the space of smooth vector fields on $M$. Given a Riemannian metric $g=\langle\cdot,\cdot\rangle$ on $M$, the \emph{Levi-Civita connection}\,Ê$D:\mathcal{X}(M)\times\mathcal{X}(M)\to\mathcal{X}(M)$ is the unique torsion-free connection compatible with the metric. This means that $D$ satisfies the equations
\begin{eqnarray*}
D_X Y -D_Y X = [X,Y], & & (zero\ torsion) \\
X\big(g(Y,Z)\big)=g(D_XY,Z)+g(Y,D_XZ),& & (compatibility\ with\ the\ metric)
\end{eqnarray*}
for all $X,Y,Z\in\mathcal{X}(M)$. The value of $D_XY$ at a point $p\in M$ depends only on $X_p$ and on the values of $Y$ along some curve tangent to $X_p$. In particular, given a smooth curve $\gamma:I\subset \mathbb{R}\to M$ the expression $D_{\dot \gamma}\dot\gamma$, also written $D_t\dot \gamma$, defines a vector field along $\gamma$. We say that $\gamma$ is a \emph{geodesic} if 
\begin{equation}\label{eq:geodesics}
D_t\dot\gamma=0.
\end{equation}
The identity $\frac d {dt}\frac 1 2|\dot\gamma|^2=\langle D_t\dot\gamma,\dot\gamma\rangle$ shows that a geodesic has constant speed $|\dot\gamma|$. Two geodesics have the same image if and only if their parametrizations differ by an affine transformation of $\mathbb{R}$. 

Equation~\eqref{eq:geodesics} is a second order non-linear ordinary differential equation (ODE) with smooth coefficients. As a consequence of general existence and uniqueness theory, together with smooth dependence on the initial conditions, one can define the \emph{exponential map at $p\in M$},
$$
\exp_p:\mathcal{V}(0)\subseteq T_p M\to M.
$$
Here $\mathcal{V}(0)$ denotes a sufficiently small open neighborhood of $0\in T_pM$. By definition, the curve $t\mapsto \exp_p(tX)$ is the unique geodesic passing through $p$ at time $t=0$ with speed $X$. The exponential map is a local diffeomorphism since $d\exp_p(0)=\textrm{Id}$, and this implies that $p$ is connected to any nearby point by a unique ``short" geodesic. As a matter of fact, a much stronger statement is true: any point $p$ has a basis of geodesically convex neighborhoods, i.e. open sets $\mathcal{U}$ such that any two points in $\mathcal{U}$ are connected by a unique geodesic contained in $\mathcal{U}$. 

\smallskip
\noindent {\bf Examples.} (i)~Denote by $E^{n}$ the Euclidean $n$-dimensional space. Geodesics in $E^{n}$ are straight lines parametrized as affine embeddings of $\R$. 

(ii)~A curve $\gamma$ lying on a submanifold $M\subset E^n$ is geodesic for the induced Riemannian metric if and only if the acceleration vector field $\ddot\gamma$ is orthogonal to $M$. 

(iii)~On the sphere $S^n=\{x\in E^{n+1}\, :\, \|x\|=1\}$ endowed with the induced metric, all the geodesics close up: their images are the great circles on $S^n$, i.e. the circles obtained by intersecting $S^n$ with $2$-dimensional vector subspaces of $\R^{n+1}$. 

(iv)~On the complex projective space $\C P^n=S^{2n+1}/S^1$, which we view as the quotient of the unit sphere $S^{2n+1}\subset \C^{n+1}$ by the diagonal action of $S^1=U(1)$ and which we endow with the quotient (Fubini-Study) metric, the complex lines $\C P^1\subset \C P^n$ are totally geodesic submanifolds, isometric to the $2$-sphere of radius $1/2$ in $\R^3$. The image of a geodesic starting at a point $p$ in the direction $v$ is a great circle on the unique complex line through $p$ tangent to $v$, and in particular all geodesics on $\C P^n$ close up.
\smallskip

\noindent We use for the \emph{Riemannian curvature tensor} $R:T_pM\times T_p M\to \mathrm{End}(T_pM)$ the sign convention 
$$
R(X,Y)=-[D_X,D_Y]+D_{[X,Y]},
$$
i.e. $R(X,Y)Z=D_YD_XZ-D_XD_YZ+D_{[X,Y]}Z$ for all $X,Y,Z\in T_p M$. The Riemannian curvature tensor is the fundamental invariant of a Riemannian metric. Let us only mention here that it takes values in the space of anti-symmetric endomorphisms of $T_pM$, i.e. in the Lie algebra of orthogonal transformations of $T_pM$.

\smallskip

\noindent {2. \it Spaces of paths and energy functional.} Let $p,q\in M$ be two distinct points. We consider the space 
$$
\mathcal{P}(p,q)=\{\gamma\in C^\infty([0,1],M)\, : \, \gamma(0)=p,\ \gamma(1)=q\}
$$
of smooth paths (strings) defined on $[0,1]$ and running from $p$ to $q$. 
The space $\mathcal{P}(p,q)$ is a Fr\'echet manifold and the tangent space at a path $\gamma$ is 
$$
T_\gamma\mathcal{P}(p,q)=\{\xi\in\Gamma(\gamma^*TM)\, : \, \xi(0)=0,\ \xi(1)=0\}.
$$
More precisely, an element $\xi\in T_\gamma\mathcal{P}(p,q)$ determines a curve 
\begin{equation}\label{eq:cxi}
c_\xi:\mathcal{V}(0)\subset \mathbb{R}\to\mathcal{P}(p,q), \qquad c_\xi(s)(t):=\exp_{\gamma(t)}(s\xi(t)),
\end{equation}
which satisfies $c_\xi(0)=\gamma$ and $\frac d {ds}\big|_{s=0} c_\xi(s)(t)=\xi(t)$, $t\in[0,1]$.   
  
The \emph{energy functional} 
$$
E:\mathcal{P}(p,q)\to\mathbb{R}
$$
is defined by 
\begin{equation} \label{eq:E} 
E(\gamma):=\int_0^1|\dot \gamma|^2.
\end{equation}
The differential of $E$ at $\gamma$ is the linear map $dE(\gamma):T_\gamma\mathcal{P}(p,q)\to\R$ given by 
\begin{equation}\label{eq:dE}
dE(\gamma)\xi=2\int_0^1\langle D_t\xi,\dot\gamma\rangle=-2\int_0^1\langle\xi,D_t\dot\gamma\rangle.
\end{equation}
This is called \emph{the first variation formula}. It shows that 
\begin{center}
\emph{$\gamma$ is a critical point of $E$ iff $D_t\dot\gamma=0$, i.e. $\gamma$ is a geodesic from $p$ to $q$}. 
\end{center}
To prove the first equality in~\eqref{eq:dE}, recall the definition of $c_\xi(s)(t)=:c_\xi(s,t)$ above and compute
\begin{eqnarray*}
dE(\gamma)\xi=\frac d {ds}\Big|_{s=0} E(c_\xi(s)) & = & 2\int_0^1\langle D_s \partial_t c_\xi,\partial_t c_\xi\rangle \\Ê
& = & 2\int_0^1\langle D_t \partial_s c_\xi,\partial_t c_\xi\rangle = 2\int_0^1\langle D_t\xi,\dot\gamma\rangle,
\end{eqnarray*}
using that $D_s\partial_t=D_t\partial_s$. The second equality in~\eqref{eq:dE} follows by integrating by parts and using the vanishing condition at the endpoints for $\xi$. This integration by parts is exactly the procedure used to derive the general \emph{Euler-Lagrange equations}, which in our case read $D_t\dot\gamma=0$. 
 
In order to avoid the subtleties of analysis on Fr\'echet manifolds we switch to a Banach -- and actually Hilbert -- setup and consider as a domain of definition for the energy functional the space 
$$
\Omega(p,q) = \{\gamma\in H^1([0,1],M) \, : \, \gamma(0)=p,\ \gamma(1)=q\}
$$
of paths of Sobolev class $H^1=W^{1,2}$ defined on $[0,1]$ and running from $p$ to $q$. Note that such paths are necessarily continuous and therefore the condition at the endpoints makes sense. The space $\Omega(p,q)$ is a smooth Hilbert manifold and the tangent space at $\gamma$ is 
$$
T_\gamma\Omega(p,q)=\{\xi\in H^1(\gamma^*TM)\, : \, \xi(0)=0,\ \xi(1)=0\},
$$
the space of vector fields along $\gamma$ which are of Sobolev class $H^1$ and vanish at the endpoints. 
The energy functional $E:\Omega(p,q)\to\R$ is of class $C^2$, and the formulas expressing $dE(\gamma)$ remain the same. Critical points are now $H^1$-solutions of $D_t\dot\gamma=0$ and, this being an elliptic equation, they are necessarily smooth and are therefore geodesics.

\begin{remark}[On two other functionals] The energy functional will be our main tool for studying geodesics. There are two other important functionals of geometric origin whose critical points are related to geodesics. 

The first one is the \emph{length functional}
$$
L:\mathcal{P}(p,q)\to \R, \qquad L(\gamma):=\int_0^1|\dot \gamma|.
$$
The change of variables formula shows that $L$ is invariant under the action of the (infinite-dimensional) group of diffeomorphisms of the interval $[0,1]$. As a consequence, critical points of $L$ come in infinite-dimensional families. This degeneracy can be removed using the observation that any path $\gamma\in\mathcal{P}(p,q)$ admits a unique positive reparametrization on $[0,1]$ with constant speed. The reader can then prove that $\gamma\in\mathcal{P}(p,q)$ is a geodesic if and only if it has constant speed and is a critical point of $L(\gamma)$. Note that $L$ is only differentiable at paths $\gamma$ such that $\dot\gamma\neq 0$. As such, it is not well adapted to the study of geodesics from a variational point of view. Note that the Cauchy-Schwarz inequality implies 
$$
L(\gamma)\leq \sqrt{E(\gamma)},
$$
with equality if and only if $\gamma$ is parametrized proportional to arc-length (PPAL). 

The second functional is the \emph{norm functional}
$$
F:\mathcal{P}(p,q)\to\R,\qquad F(\gamma):=\sqrt{E(\gamma)}.
$$
This functional is perhaps best adapted for the variational study of geodesics: if $p\neq q$ it is everywhere differentiable and has the same differentiability class as the energy. Moreover, it is additive under energy-minimizing concatenation of paths and its minimax values behave well with respect to products of Chas-Sullivan type. We shall not use these features here and refer to~\cite{Goresky-Hingston,Hingston-Oancea} for applications that use these features in an essential way. The norm functional was first used in the study of geodesics by Goresky and Hingston in~\cite{Goresky-Hingston}. 
\end{remark}

\noindent {3. \it Spaces of loops.} The above discussion has a \emph{periodic} counterpart. Let us use $S^1=\mathbb{R}/\mathbb{Z}$ as a model for the circle and consider
$$
\mathcal{L} M := C^\infty(S^1,M),
$$
the space of smooth \emph{loops} in $M$. This is a Fr\'echet manifold on which $O(2)=SO(2)\rtimes\{\pm1\}$
acts naturally by $(e^{i\theta}\cdot\gamma)(t):=\gamma(t+\theta)$ and $(-1\cdot\gamma)(t):=\gamma(-t)$.
Again, it is more convenient to work with the smooth Banach manifold 
$$
\Lambda M :=H^1(S^1,M)
$$
of loops of Sobolev class $H^1$. We shall refer to $\Lambda M$ as the \emph{free loop space} of $M$. Obviously $\mathcal{L}M\subset \Lambda M$, and the $O(2)$-action extends naturally to $\Lambda M$. 
The inclusion $\mathcal LM\hookrightarrow \Lambda M$ is a homotopy equivalence~\cite[Theorem~13.14]{Palais-Foundations}. 

Equation~\eqref{eq:E} defines a smooth functional $E:\Lambda M\to\mathbb{R}$, whose critical points are smooth periodic (or closed) geodesics. In the present situation $E$ is $O(2)$-invariant and this forces some degeneracy for the critical points. A first rough (and binary) classification of closed geodesics is the following: on the one hand we have the constant ones, corresponding to points in $M$, and on the other hand we have the non-constant ones. The isotropy group at a constant geodesic is $O(2)$, whereas the isotropy group at a non-constant one is $\mathbb{Z}/k\mathbb{Z}$, where $1/k$, $k\in\Z^+$ is its minimal period. Non-constant geodesics come in pairs resulting from reversing the time.

The tangent space at $\gamma\in\Lambda M$ is $T_\gamma\Lambda M=H^1(\gamma^*TM)$, the space of sections of $\gamma^*TM$ of Sobolev class $H^1$. \emph{In the sequel we  view $\Lambda M$ as a Hilbert manifold with respect to the $H^1$-scalar product}
$$
\langle \xi,\eta\rangle_1 := \int_{S^1} \langle \xi,\eta\rangle + \int_{S^1} \langle D_t\xi,D_t\eta\rangle = \langle\xi,\eta\rangle_0 + \langle D_t\xi,D_t\eta\rangle_0.
$$

As a consequence of the Arzel\`a-Ascoli theorem we have a compact inclusion $\Lambda M=H^1(S^1,M)\hookrightarrow C^0(S^1,M)$. This in turn can be used to prove the following.

\begin{proposition}[{\cite[1.4.5]{Klingenberg-book}}]
The Riemannian metric on $\Lambda M$ given by the $H^1$-scalar product is complete. 
\end{proposition} 

\begin{remark}[On the choice of completion] \label{rmk:Floer} The choice of completion for the manifold $\mathcal{L} M$ is crucial for applications. The $H^1$-completion for the space of free loops ensures that the energy functional has a well-defined gradient flow on $\Lambda M$, which moreover satisfies the ``Palais-Smale condition", an infinite-dimensional analogue of properness (see~\S\ref{sec:Morse} below).

A contrasting example coming from symplectic geometry is the following. Consider the standard phase space $(\R^{2n},\omega=dp\wedge dq)$ endowed with the standard complex structure $J$ and the Euclidean metric $\omega(\cdot, J \cdot)$. Given a periodic Hamiltonian $H:S^1\times M\to \R$, the Hamiltonian action functional $A_H(\gamma):=\int_\gamma pdq - Hdt$ defined on the space of free loops is not well-adapted to variational methods and Morse theory since the index and the coindex of a critical point are both infinite. We recall that the critical points of this functional are the periodic orbits of the Hamiltonian system $\dot q= \p H/\p p$, $\dot p = - \p H/\p q$. One of Floer's insights~\cite{Floer} was to consider on $\mathcal L M$ the $L^2$-scalar product. Although the  equation of gradient lines for $A_H$ with respect to this metric is not integrable, i.e. the $L^2$-gradient flow does not exist in the ODE sense, the same equation interpreted as an equation on the cylinder $\R\times S^1$ turns out to be a $0$-order perturbation of the Cauchy-Riemann equation, i.e. an elliptic PDE (partial differential equation). The reader is referred to~\cite{Audin-Damian} 
for an account of Floer's theory. 
\end{remark}

\section{Morse theory} \label{sec:Morse}

In this section we explain the rudiments of Morse theory, focusing on the space of loops $\Lambda M$. The discussion can be adapted in a straightforward way to the space of paths $\Omega(p,q)$ (which is the setup of Milnor's classical book on the subject~\cite{Milnor-Morse_theory}). 

The gradient $\nabla E$ is the vector field on $\Lambda M$ defined by 
$$
\langle \nabla E(\gamma),\xi\rangle_1=dE(\gamma)\xi = \langle \dot\gamma, D_t\xi\rangle_0, \qquad \xi\in H^1(\gamma^*TM).
$$
If $\gamma$ is smooth we have $\langle \dot\gamma,D_t\xi\rangle_0=-\langle D_t\dot\gamma, \xi\rangle_0$, so that 
$\nabla E(\gamma)\in\Gamma(\gamma^*TM)$ is the unique (periodic) solution of 
$$
D_t^2\eta(t)-\eta(t)=D_t\dot\gamma(t).
$$

The following property of the energy functional, due to Palais and Smale, is the crucial ingredient that makes Morse theory work in an infinite dimensional setting~\cite{Palais}. 

\begin{theorem}[\cite{Palais}, \cite{Eliasson}, {\cite[1.4.7]{Klingenberg-book}}] \label{thm:PS}
The energy functional $E:\Lambda M\to \R$ satisfies condition (C) of Palais and Smale~\cite{Palais-Smale}:
\begin{center} \label{conditionC}
\emph{(C)} \quad Let $(\gamma_m)\in\Lambda M$ be a sequence such that $E(\gamma_m)$ is bounded and $\|\nabla E(\gamma_m)\|_1$ tends to zero. Then $(\gamma_m)$ has limit points and every limit point is a critical point of $E$.
\end{center}
\end{theorem}

\noindent For the proof of Theorem~\ref{thm:PS} one first produces a $C^0$-limit using the Arzel\`a-Ascoli theorem. This allows one to work in a local chart around a smooth approximation of the limit and use the completeness of $\Lambda M$.

Let $\mathrm{Crit}(E)$ denote the set of critical points of the energy functional, i.e. the set of closed geodesics on $M$. Given $a\ge 0$ we denote 
$$
\Lambda^{\le a}:=\{\gamma\, : \, E(\gamma)\le a\},
$$ 
$$
\Lambda^{< a}:=\{\gamma\, : \, E(\gamma)<a\},
$$ 
and 
$$
\Lambda^a:=\{\gamma\, : \, E(\gamma)= a\}.
$$ 
The following are direct consequences of the Palais-Smale condition (C) (details can be found in~\cite[\S1.4]{Klingenberg-book}). 

\smallskip
\noindent \emph{Properties of the negative gradient flow $\frac d {ds} \phi_s=-\nabla E(\phi_s)$ of the energy functional.}
\renewcommand{\theenumi}{\roman{enumi}}
\begin{enumerate}
\item \label{item:i} $\mathrm{Crit}(E)\cap \Lambda^{\le a}$ is compact for all $a\ge 0$.
\item \label{item:ii} the flow $\phi_s$ is defined for all $s\ge 0$.
\item \label{item:iii} given an interval $[a,b]$ of regular values, there exists $s_0\ge 0$ such that $\phi_s(\Lambda^{\le b})\subset \Lambda^{\le a}$ for all $s\ge s_0$. 
\item \label{item:iv} $\Lambda^0\equiv M\subset \Lambda^{\le \eps}$ is a strong deformation retract via $\phi_s$ for $\eps>0$ small enough.
\end{enumerate}

Let $\gamma\in\mathrm{Crit}(E)$ be a critical point. Because of the $O(2)$-invariance of $E$, the entire orbit $O(2)\cdot\gamma$ is contained in $\mathrm{Crit}(E)$. The \emph{index} $\lambda(\gamma)$ of $\gamma$ is by definition the dimension of the negative eigenspace of the Hessian $d^2E(\gamma)$. The \emph{nullity} $\nu(\gamma)$ of $\gamma$ is by definition the dimension of the null space of $d^2E(\gamma)$. Note that $\nu(\gamma)\ge 1$ if $\gamma$ is a non-constant geodesic since the critical set of $E$ is invariant under the natural $S^1$-action by reparametrization at the source.\footnote{The reader will encounter in the literature also a different convention which defines the nullity to be equal to $\nu(\gamma)-1$. This accounts for the fact that the element $\langle \dot\gamma\rangle$ in the kernel of $d^2E(\gamma)$ does not contain any geometric information. With our convention, the nullity of a geodesic that belongs to a Morse-Bott nondegenerate critical manifold (see p.~\pageref{page:Morse-Bott} below) is equal to the dimension of that manifold. The drawback is that, in the Bott iteration formulas~\eqref{eq:LN} on p.~\pageref{eq:LN}, one has to define in a slightly unnatural way the value of the function $N(z)$ at $z=1$.}

The Hessian of $E$ at $\gamma$ is expressed by the \emph{second variation formula} 
\begin{equation}\label{eq:d2E}
 d^2E(\gamma)(\xi,\eta)=-\int\langle \xi,D^2_t\eta+R(\dot\gamma,\eta)\dot\gamma\rangle.
\end{equation}
Here we use is the $L^2$-inner product, and not the $H^1$-inner product used to define the flow. Note that the Hessian is independent of the inner product, the issue of completion of $\mathcal L M$ put aside. It is useful to recall at this point that the role of the $H^1$-inner product is to guarantee the existence of the gradient flow. 
 
An element $\eta\in T_\gamma\Lambda M$ is called a \emph{Jacobi vector field} if it solves the equation 
\begin{equation} \label{eq:Jacobi}
D^2_t\eta+R(\dot\gamma,\eta)\dot\gamma=0.
\end{equation}
A variation of a geodesic $\gamma$ within the space of geodesics naturally defines a Jacobi vector field along $\gamma$. Conversely, any Jacobi vector field can be obtained in this way. Jacobi vector fields form a vector space, whose  dimension is by definition the nullity $\nu(\gamma)$. In the closed case the nullity satisfies $1\le \nu(\gamma)\le 2n-1$~: the lower bound follows from the fact that $\dot\gamma$ is a Jacobi vector field (see above), while the upper bound follows from the fact that a solution $\eta$ of the 2nd order ODE~\eqref{eq:Jacobi} is uniquely determined by the pair $(\eta(0),D_t\eta(0))$: the pair $(\dot\gamma(0),0)$ gives rise to the Jacobi field $\dot\gamma$, whereas the pair $(0,\dot\gamma)$ gives rise to the vector field $t\dot\gamma(t)$ which is not $1$-periodic and so does not belong to $T_\gamma\Lambda M$.

A closed geodesic $\gamma$ is also a critical point of the energy functional defined on the space $\Omega_p=\Omega(p,p)$ of loops based at $p=\gamma(0)$, and as such has a well-defined index, denoted $\lambda_\Omega(\gamma)$ and called \emph{the $\Omega$-index of $\gamma$}. The $\Omega$-index is expressed by Morse's famous index theorem, which we state now. In order to prepare the statement we introduce the following terminology. Given a geodesic $\gamma:[t_0,t_1]\to M$, we say that \emph{$p_0=\gamma(t_0)$ is conjugate to $p_1=\gamma(t_1)$ along $\gamma$} if there exists a nonzero Jacobi vector field $\eta$ along $\gamma$ such that $\eta(t_0)=0$, $\eta(t_1)=0$. We say that \emph{$p_0$ is conjugate to $p_1$ along $\gamma$ with multiplicity $m$} if the dimension of the space of such Jacobi vector fields is equal to $m$. 

\begin{theorem}[Morse's index theorem~{\cite[Theorem~15.1]{Milnor-Morse_theory}}]
The $\Omega$-index $\lambda_\Omega(\gamma)$ of the geodesic $\gamma:[0,1]\to M$ is equal to the number of points $\gamma(t)$, with $0<t<1$ such that $\gamma(t)$ is conjugate to $\gamma(0)$ along $\gamma$, each such conjugate point being counted with its multiplicity. The $\Omega$-index $\lambda_\Omega(\gamma)$ is finite. 
\end{theorem}

\smallskip

Ballmann, Thorbergsson and Ziller proved in~\cite{BTZ1982} that the following inequality holds
$$
\lambda_\Omega(\gamma)\le \lambda(\gamma)\le \lambda_\Omega(\gamma)+n-1.
$$
The quantity $\lambda(\gamma)-\lambda_\Omega(\gamma)$ is called the \emph{concavity} of $\gamma$; it depends on the structure of the Poincar\'e return map, i.e. the time one linearization of the geodesic flow along $\gamma$ (see~\cite[\S1]{BTZ1982}, \cite{Ziller-Katok} and the discussion below).

\smallskip

\begin{example}Ê\label{exple:indexSn} Let $M=S^n$ be the sphere with the round metric. The closed geodesics are the great circles $\gamma_k$ covered $k\ge 1$ times. Given such a circle starting at $p\in S^n$ there are $2k-1$ conjugate points $\gamma_k(t)$ with $0<t<1$: the antipode of $p$ appears $k$ times whereas $p$ itself appears $k-1$ times. Each conjugate point has multiplicity $n-1$, corresponding to the $n-1$-dimensional space of Jacobi fields obtained by letting the geodesic $\gamma_k$ vary within the space of great circles. Thus the $\Omega$-index of $\gamma_k$ is $\lambda_\Omega(\gamma_k)=(2k-1)(n-1)$, $k\ge 1$. It turns out that the concavity of $\gamma_k$ is zero, so that $\lambda(\gamma_k)=(2k-1)(n-1)$ (this is proved in~\cite[Theorem~4]{Ziller} and follows also from the discussion in~\cite[\S1]{BTZ1982}). The nullity of $\gamma_k$ is maximal, i.e. equal to $2n-1$: the value $2n-1$ is a general upper bound in dimension $n$, but for the sphere it is also a lower bound since a geodesic $\gamma_k$ lives naturally in a family of dimension $2n-1$ parametrized by the unit tangent bundle of $S^n$. 
\end{example}

\smallskip 

\begin{example} \label{exple:indexCPn} Let $M=\C P^n$ be the complex projective space, endowed with the Fubini-Study metric induced from the round metric on the unit sphere $S^{2n+1}\subset \C^{n+1}$ by viewing $\C P^n$ as $S^{2n+1}/S^1$. The complex lines $\C P^1\subset \C P^n$ are totally geodesic and isometric to round $2$-spheres of constant curvature $4$. The closed geodesics on $\C P^n$ are the great circles on these $2$-spheres covered $k\ge 1$ times, and denoted $\gamma_k$. Given such a circle $\gamma_k$ starting at $p\in\C P^n$ in the direction $v$, denote $\ell_{p,v}\subset \C P^n$ the unique complex line which passes through $p$, which is tangent to $v$, and which contains $\gamma_k$. There are $2k-1$ points $\gamma_k(t)$, $0<t<1$ which are conjugate to $p$ along $\gamma_k$: the antipode of $p$ on $\ell_{p,v}$ appears $k$ times, whereas $p$ itself appears $k-1$ times. The antipode of $p$ has multiplicity $1$, corresponding to the $1$-dimensional space of Jacobi fields given by letting the geodesic $\gamma_k$ vary within the space of great circles on $\ell_{p,v}$. The point $p$ has multiplicity $2n-1$, corresponding to the $2n-1$-dimensional space of Jacobi fields given by letting the geodesic $\gamma_k$ vary within the space of all geodesics passing through $p$, naturally parametrized by the unit tangent fiber of $\C P^n$ at $p$. The $\Omega$-index of $\gamma_k$ is therefore equal to $\lambda_\Omega(\gamma_k)= k+(k-1)(2n-1)= 2(k-1)n+1$. It turns out that the concavity of $\gamma_k$ is zero, so that $\lambda(\gamma_k)=2(k-1)n+1$~\cite[Theorem~4]{Ziller}. The nullity of $\gamma_k$ is maximal, i.e. equal to $4n-1$: this is a general upper bound in dimension $2n$, but for $\C P^n$ it is also a lower bound since every closed geodesic lives naturally in a family of dimension $4n-1$ parametrized by the unit tangent bundle of $\C P^n$.
\end{example}

Morse theory in its simplest form describes the relationship between the topology of a manifold and the structure of the critical set of a function defined on the manifold~\cite[Part~I]{Milnor-Morse_theory}. A $C^2$-function defined on a Hilbert manifold is said to be a \emph{Morse function} if all its critical points are non-degenerate, meaning that the Hessian has a zero-dimensional kernel at each critical point. This can never happen for the energy functional defined on $\Lambda M$ for two reasons. One the one hand, the critical points at level $0$, i.e. the constant geodesics, form a manifold of dimension $\dim\, M$ and can therefore never be non-degenerate. On the other hand, the energy functional is $S^1$-invariant, so that the kernel of $d^2E(\gamma)$ at a non-constant geodesic $\gamma$ is always at least $1$-dimensional since it contains the infinitesimal generator of the action, which is the vector field $\dot\gamma$ along $\gamma$. Note that neither of these issues arises if one studies the energy functional on the space of paths with fixed and distinct endpoints. 
 
However, the energy functional can successfully be studied by the methods of Morse theory as generalized by Bott in~\cite{Bott1954}. A $C^2$-function defined on a Hilbert manifold is said to be a \emph{Morse-Bott function} \label{page:Morse-Bott} if its critical set is a disjoint union of closed (connected) submanifolds and, for each critical point, the kernel of the Hessian at that critical point coincides with the tangent space to (the relevant connected component of) the critical locus. In this case we say that the critical set is \emph{non-degenerate}. The \emph{index} of a critical point is defined as the dimension of a maximal subspace on which the Hessian is negative definite. The \emph{nullity} of a critical point is defined as the dimension of the kernel of the Hessian. The index and nullity are constant over each connected component of the critical set. 

In the case of the energy functional $E:\Lambda M\to \R_+$, the critical manifold of absolute minima at zero level, consisting of constant loops, is always non-degenerate~\cite[Proposition~2.4.6]{Klingenberg-book} and obviously has index $0$. It turns out that, for a generic choice of metric, one can achieve that all the critical orbits $S^1\cdot\gamma$ corresponding to closed non-constant geodesics are non-degenerate. The non-degeneracy condition means in this particular case that $\ker\,d^2E(\gamma)=\langle\dot\gamma\rangle$ and $d^2E(\gamma)$ is non-degenerate on $V:=\langle\dot\gamma\rangle^\perp$. Such metrics are sometimes referred to as \emph{bumpy metrics}. We shall see in~\S\ref{sec:GM} that the index of any closed geodesic is finite, as a consequence of the ellipticity of the linearization of the equation of closed geodesics. 

In case the manifold $M$ admits a metric with a large group of symmetries, the energy functional naturally admits higher-dimensional critical manifolds, and in good situations these are non-degenerate. We shall see two such explicit instances for spheres and complex projective spaces in~\S\ref{sec:completing}.

\smallskip

\begin{theorem}[\cite{Palais}, \cite{Bott1954}] \label{thm:Morse}
Assume that the Riemannian metric on $M$ is chosen such that the energy functional $E$ on $\Lambda M$ is Morse-Bott. 
\begin{enumerate}
\item The critical values of $E$ are isolated and there are only a finite number of connected components of $\mathrm{Crit}(E)$ on each critical level.
\item The index and nullity of each connected component of $\mathrm{Crit}(E)$ are finite. 
\item If there are no critical values of $E$ in $[a, b]$ then $\Lambda^{\le b}$ retracts onto (and is actually diffeomorphic to) $\Lambda^{\le a}$.
\item \label{item:crossing} Let $a<c < b$ and assume $c$ is the only critical value of $E$ in the interval $[a, b]$. Denote $N_1,\dots,N_r$ the connected components of $\mathrm{Crit}(E)$ at level $c$, and denote $\lambda_1,\dots,\lambda_r$ their respective indices. 
\begin{itemize}
\item Each manifold $N_i$ carries a well-defined vector bundle $\nu^-N_i$ of rank $\lambda_i$ consisting of negative directions for $d^2E|_{N_i}$.  
\item The sublevel set $\Lambda^{\le b}$ retracts onto a space homeomorphic to $\Lambda^{\le a}$ 
with the disc bundles $D\nu^-N_i$ disjointly attached to $\Lambda^{\le a}$ along their boundaries.
\end{itemize} 
\end{enumerate} 
\end{theorem}

In the statement of the theorem we have denoted by $D\nu^-N_i$ the disc bundles associated to some fixed scalar product on the fibers of $\nu^-N_i$. Of course, one can think of $\nu^-N_i$ as being a sub-bundle of $TN_i^\perp$ with induced scalar product coming from the Hilbert structure on $\Lambda M$. The meaning of disjointly attaching $D\nu^-N_i$ to $\Lambda^{\le a}$ along the boundary is the following: there exist smooth embeddings $\varphi_i:\p D\nu^-N_i\to\p\Lambda^{\le a}$ with disjoint images, with respect to which one can form the quotient space $\Lambda^{\le a}\cup\bigcup_i D\nu^-N_i/\sim$, where a point in $\p D\nu^-N_i$ is identified with its image in $\p\Lambda^{\le a}$ via $\varphi_i$. Note that we actually have $\p\Lambda^{\le a}=\Lambda^a$.

Figure~\ref{fig:Morse} provides an intuitive explanation for~\eqref{item:crossing} in the above theorem, in the case of a Morse function $f$ defined on a finite-dimensional manifold. In the neighborhood of a critical point of index $\lambda$,  the local model for $f$ is provided by the quadratic form $(x_1,\dots,x_n)\mapsto c-x_1^2-\dots-x_\lambda^2+x_{\lambda+1}^2+\dots+x_n^2$ defined in a neighborhood of $0\in\R^n$. As one crosses the critical level $c$, the sublevel $c+\eps$ with $\eps>0$ small enough retracts onto the union of the sublevel $c-\eps$ and of a $\lambda$-dimensional cell. The retraction is provided by a suitable modification of the negative gradient flow of $f$~\cite[Part~I, \S3]{Milnor-Morse_theory}. To relate this picture to the above theorem, one should interpret this $\lambda$-dimensional cell as the negative vector bundle over the critical manifold which consists of a point. 
In the Morse-Bott case, this whole picture has to be thought of in a family parametrized by the connected component $N_i$ of the critical locus which contains the critical point. This is the reason for the appearance of the vector bundles of rank $\lambda_i$ in the previous theorem. The retraction and the gluing maps are again provided by suitable modifications of the negative gradient flow. 

\begin{figure}[htp]      
\centering       
\includegraphics[scale=0.8]{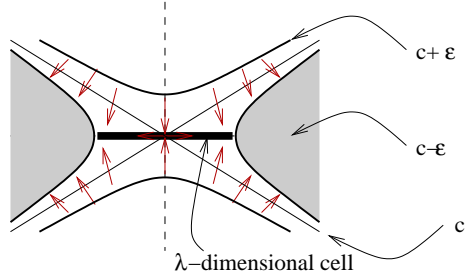}      
\caption{{Crossing a critical value in a local model.}}\label{fig:Morse}      
\end{figure} 

The statement of~\eqref{item:crossing} can be further enhanced as follows.

\begin{corollary}
Under the assumptions and notations of~\eqref{item:crossing} in the above theorem, the following hold true:
\begin{itemize}
\item the sublevel set $\Lambda^{\le b}$ retracts onto $\Lambda^{\le c}$,
\item the sublevel set $\Lambda^{< c}$ retracts onto $\Lambda^{\le a}$.
\end{itemize}
\end{corollary}

The above geometric theorem can be turned into an effective \emph{iterative} way of handling the homology groups of $\Lambda M$ using the increasing filtration $\Lambda^{\le a}$, with respect to which we have $\displaystyle \Lambda M = \lim_{\stackrel \longrightarrow {a\to\infty}}\Lambda^{\le a}$, and thus $\displaystyle H_\cdot(\Lambda M)=\lim_{\stackrel \longrightarrow {a\to\infty}} H_\cdot(\Lambda^{\le a})$. Indeed, assuming $E$ is non-degenerate (which includes the case of a generic metric) the sequence $0=c_0<c_1<c_2<\dots$ of its critical values is diverging and the following hold. 
\begin{itemize}
\item if $c_i<a<c_{i+1}$ then the inclusion $\Lambda^{\le c_i}\hookrightarrow \Lambda^{\le a}$ induces an isomorphism 
$$
H_\cdot(\Lambda^{\le c_i})\stackrel \sim \to H_\cdot(\Lambda^{\le a}).
$$
In particular the value of $H_\cdot(\Lambda^{\le a})$ ``jumps'' only when $a$ crosses a critical value. 
\item the change in homology upon passing a critical level $c$ is encoded in the homology long exact sequence of the pair $(\Lambda^{\le c},\Lambda^{< c})$, i.e. 
\begin{equation} \label{eq:lespair}
\dots\longrightarrow H_\cdot(\Lambda^{<c})\longrightarrow H_\cdot(\Lambda^{\le c})\longrightarrow H_\cdot(\Lambda^{\le c},\Lambda^{< c})\stackrel {[-1]} \longrightarrow\dots
\end{equation}
Denoting by $N^c:=\sqcup_{i=1}^r N_i$ the critical set at level $c$, we have
\begin{eqnarray*}
H_\cdot(\Lambda^{\le c},\Lambda^{<c}) & \simeq & H_\cdot(\sqcup_i D\nu^-N_i,\sqcup_i \p D\nu^-N_i) \\
& \simeq & \oplus_i H_{\cdot-\lambda_i}(N_i;o_{\nu^-N_i}).
\end{eqnarray*}
Here the first isomorphism is induced by excision and follows from~\eqref{item:crossing} in the previous theorem. The second isomorphism is the general form of the Thom isomorphism for finite rank vector bundles which are not necessarily assumed to be orientable. The notation $H_\cdot(N_i;o_{\nu^-N_i})$ stands for the homology groups of $N_i$ with coefficients in the local system of orientations of the bundle $\nu^-N_i$~\cite{Bott1954}. Alternatively, this is the homology of $N_i$ with coefficients in the flat line bundle $\det\, \nu^-N_i$. In particular, we see that the change in topology upon crossing the level $c$ is determined solely by local data along the non-degenerate critical set of $E$. 
\end{itemize} 

\begin{remark}[On local coefficients] \label{rmk:invisible} We refer to~\cite[\S5.3]{McC} for a discussion of homology with local coefficients, as well as to the original paper of Steenrod~\cite{Steenrod}. Rather than making explicit the definition, we shall content ourselves with spelling out an example, namely that of a closed geodesic which is ``homologically invisible''. 
Assume that $\gamma$ is a non-constant and non-degenerate closed geodesic of index $\lambda$, so that $S^1\cdot\gamma$ is a Morse-Bott non-degenerate critical manifold. Assume further that the negative bundle $\nu^-\to S^1\cdot\gamma$ is not orientable. Then we have $H_\cdot(S^1\cdot\gamma;o_{\nu^-})=0$ if we use coefficients in a ring where $2$ is invertible. Indeed, computing the homology using a cellular decomposition of $S^1\cdot\gamma$ with two cells denoted $e_0$ (in dimension $0$), respectively $e_1$ (in dimension $1$), we have $\p e_1= e_0 - \epsilon e_0$, where $\epsilon\in\{\pm 1\}$ is the sign given by the monodromy of the orientation local system around the loop $S^1\cdot\gamma$. In the non-orientable case this sign is indeed equal to $-1$, so that $\p e_1=2 e_0$. We refer to such a closed geodesic as being ``homologically invisible'' since, if $S^1\cdot\gamma$ was the only critical manifold at level $c$, the homology with coefficients in a ring in which $2$ is invertible would not change upon crossing the critical level.

These geodesics play a distinctive role in Rademacher's resonance formula~\cite[Theorem~3]{Rademacher1989}. In the notation of that paper, they are responsible for the appearance of coefficients $\gamma_c\in\{\pm \frac 1 2\}$ in the resonance formula. These half-integer coefficients reflect the fact that, in the tower of iterates of a simple geodesic $c$, the even iterates should not be counted since they are ``homologically invisible''. Examples of such homologically invisible geodesics are the even iterates of hyperbolic geodesics of odd index; a geodesic is said to be hyperbolic if the eigenvalues of its linearized Poincar\'e return map, which is a symplectic matrix, lie outside the unit circle (see the discussion following Theorem~\ref{thm:Bott} below).  
\end{remark}

We shall see in~\S\ref{sec:completing} that, for the homogeneous metrics on spheres or complex projective spaces, all the geodesics are homologically visible, and moreover the connecting homomorphisms in the long exact sequences of the pair $(\Lambda^{\le c},\Lambda^{<c})$ vanish for all critical values $c$. This will provide a geometric way of computing the homology groups $H_\cdot(\Lambda S^n)$ and $H_\cdot(\Lambda \C P^n)$ with arbitrary coefficients. 

\section{The Lyusternik-Fet theorem} \label{sec:Lyusternik-Fet}

In this section we address the first question asked in the introduction, namely the existence of at least one nontrivial closed geodesic on a given closed Riemannian manifold $M$. 
As already discussed, the case of non-simply connected manifolds follows directly from the properties of the gradient flow of the energy functional. 

\begin{theorem}[Cartan] \label{thm:Cartan}
Assume $M$ is closed and $\pi_1(M)$ is nontrivial. There is a closed geodesic in each nontrivial free homotopy class or, equivalently, in each nontrivial conjugacy class of $\pi_1(M)$. 
\end{theorem}

\begin{proof}
Let us fix a nontrivial free homotopy class $\alpha$, and denote the corresponding connected component of $\Lambda M$ by $\Lambda_\alpha$. Denote by  
$$
m_\alpha:=\inf\,\{E(\gamma)\, : \, \gamma\in \Lambda_\alpha\}
$$ 
the minimal energy of loops in $\Lambda_\alpha$. 

We claim that $m_\alpha$ is positive. 
Arguing by contradiction, we would find elements in $\Lambda_\alpha$ of arbitrarily small energy, hence of arbitrarily small length in view of the inequality $L\leq \sqrt E$. Because the injectivity radius of a compact manifold is strictly positive, such loops would necessarily be contained in a ball and we would conclude that $\alpha$ is the trivial free homotopy class. 

We now claim that $m_\alpha$ is a critical value of $E|_{\Lambda_\alpha}$. Any critical point on the level $m_\alpha$ is then a nontrivial closed geodesic in view of the positivity of $m_\alpha$.  To prove the claim, assume that $m_\alpha$ is a regular value. The Palais-Smale condition implies that $[m_\alpha-\eps,m_\alpha+\eps]$ is an interval of regular values for $\eps>0$ small enough. Denoting $\phi_s$, $s\ge 0$ the negative gradient flow of the energy functional, property~\eqref{item:iii} on page~\pageref{item:iii} ensures that $\phi_s(\Lambda_\alpha^{\le (m_\alpha+\eps)})\subset \Lambda_\alpha^{\le (m_\alpha-\eps)}$ for $s\gg 0$. Since $\Lambda_\alpha^{\le (m_\alpha+\eps)}$ is nonempty, we would obtain that $\Lambda_\alpha^{\le (m_\alpha-\eps)}$ is nonempty as well, but this would contradict the definition of $m_\alpha$.
\end{proof}

The positivity of $m_\alpha$ in the above proof is essential, and this fact rests on the nontriviality of the chosen free homotopy class $\alpha$. On a simply connected manifold the previous proof would break down in that it would only produce a constant geodesic. With this in view, the reader is perhaps now ready to appreciate the celebrated theorem of Lyusternik-Fet~\cite{Lyusternik-Fet}.

\begin{theorem}[Lyusternik-Fet~\cite{Lyusternik-Fet}, {\cite[Theorem~2.1.6]{Klingenberg-book}}] \label{thm:LS}
Any closed simply connected Riemannian manifold carries at least one closed geodesic. 
\end{theorem}

The proof rests on the minimax principle, which we phrase according to Klingenberg~\cite[\S2.1]{Klingenberg-book}. Denote again $\phi_s:\Lambda M\to \Lambda M$ for $s\ge 0$ the negative gradient flow of the energy functional. A \emph{flow-family} $\mathcal{A}$ is a nonempty collection of nonempty subsets of $\Lambda M$ such that
\begin{itemize}
\item $E|_A$ is bounded for every $A\in\mathcal A$;
\item $A\in\mathcal{A}$ implies $\phi_sA\in\mathcal A$ for all $s\ge 0$. 
\end{itemize}

\begin{proposition}\label{prop:minimax}
Given a flow-family $\mathcal{A}$, the number 
$$
\kappa_{\mathcal{A}} = \inf_{A\in\mathcal A} \sup_A E
$$
is a critical value. 
\end{proposition}

We call the value $\kappa_{\mathcal A}$ the \emph{minimax value of $E$ over $\mathcal A$}.

\begin{proof}
By definition $\kappa:=\kappa_{\mathcal A}<\infty$. Assume by contradiction that $\kappa$ is a regular value. By compactness of $\mathrm{Crit}(E)\cap \Lambda^{\le (\kappa+1)}$ (property~(i) on page~\pageref{item:i}), we infer the existence of some $\eps>0$ such that $[\kappa-\eps,\kappa+\eps]$ consists of regular values. By definition of $\kappa$ there exists $A\in\mathcal A$ such that $A\subset \Lambda^{\le (\kappa+\eps)}$. But in this case $\phi_s(A)\subset \Lambda^{\le (\kappa-\eps)}$ for $s$ large enough (property~(iii) on page~\pageref{item:iii}), and at the same time $\phi_s(A)\in\mathcal A$, which contradicts the definition of $\kappa$.
\end{proof}

\smallskip 

\noindent {\bf Exercise.} Convince yourself that Proposition~\ref{prop:minimax} was already used in disguise in the proof of Theorem~\ref{thm:Cartan}. 

\smallskip

\noindent \emph{Proof of the theorem of Lyusternik and Fet}. The proof can be summarized as being an application of the minimax procedure over a suitable nontrivial homotopy class. 

The canonical fibration $\Omega M\hookrightarrow \Lambda M\stackrel {\textrm{ev}} \longrightarrow M$  admits a section $M\hookrightarrow \Lambda M$ given by the inclusion of constant loops, and this implies that the associated homotopy long exact sequence reduces to split short exact sequences $0\to \pi_k\Omega M\to \pi_k\Lambda M\to \pi_k M\to 0$. We thus have canonical isomorphisms
$$
\pi_k\Lambda M\simeq \pi_k\Omega M\rtimes \pi_k M \simeq \pi_{k+1}M \rtimes \pi_k M, \qquad k\ge 1
$$
between the $k$-th homotopy group of $\Lambda M$ and the semi-direct product of the $k$-th homotopy group of $\Omega M$ by the $k$-th homotopy group of $M$.

Pick now the smallest $k\ge 1$ such that $\pi_{k+1}M=H_{k+1}(M;\mathbb{Z})\neq 0$, so that $\pi_k\Lambda M\simeq \pi_{k+1}M\neq 0$. Pick as flow-family $\mathcal A$ the set of images $f:S^k\to \Lambda M$ in a fixed nontrivial homotopy class. We claim that the associated minimax value $\kappa_{\mathcal A}$ is strictly positive. Indeed, if there existed $f$ such that $f(S^k)\subset \Lambda^{\le \varepsilon}$ for $\varepsilon>0$ arbitrarily small, by property~\eqref{item:iv} on page~\pageref{item:iv} this same $f$ would provide a nontrivial element in $\pi_k\Lambda^0=\pi_k M$, and this would in turn contradict our choice of $k\ge 1$. This proves the claim, and also the theorem, since a closed geodesic on the non-zero critical level $\kappa_{\mathcal A}$ must necessarily be non-constant. 
\hfill{$\square$}

\begin{remark}
The idea of minimization that underlies the proof of both of the above theorems can be visualized -- and rephrased -- dynamically as follows. In Theorem~\ref{thm:Cartan}, if one starts with a generic \emph{point} in the component $\Lambda_\alpha$ and lets it flow by the negative gradient, one will eventually converge to a minimum of the energy in the class $\alpha$, and that minimum will necessarily be a nonconstant geodesic. In the context of Theorem~\ref{thm:LS}, this same procedure will only produce a constant geodesic. However, if one starts with a suitable nontrivial \emph{cycle} inside $\Lambda M$ and lets it flow by the negative gradient, that cycle will eventually get ``hooked up'' at some critical point. Otherwise, the corresponding homology class would come from the homology of $M$, which is extremely poor in comparison with the homology of the free loop space. This phenomenon of ``homological abundance'' for the free loop space is exploited in a striking way in the Gromoll-Meyer theorem, see~\S\ref{sec:GM}.  
\end{remark}

The theorem of Lyusternik and Fet constituted a historical breakthrough in the calculus of variations. A similar breakthrough was represented by the following more delicate theorem of Lyusternik and Schnirelmann.

\begin{theorem}[Lyusternik-Schnirelmann~\cite{LS1,LS2},~\cite{Ballmann1978},~{\cite[Theorem~A.3.1]{Klingenberg-book}}]
On the $2$-sphere endowed with an arbitrary Riemannian metric there exist at least three distinct simple closed geodesics.  
\end{theorem}

We do not give details for the proof of this theorem here, but explain instead the ``principle of subordinated classes"~\cite[p.~343]{Bott-oldnew}, which is the new conceptual tool that Lyusternik and Schnirelmann  introduced. This allows to detect \emph{distinct} minimax values using the multiplicative structure of the cohomology of a manifold. (In the  proof of the Lyusternik-Fet theorem we used a dual homological construction.) 

Let $N$ be a Hilbert manifold and $f:N\to\R$ a $C^2$ function which is bounded from below and which satisfies the Palais-Smale condition (C) on page~\pageref{conditionC}. Denote $N^a:=\{f\le a\}$ and $j_a:N^a\hookrightarrow N$ the inclusion. Given $\alpha\in H^*(N)\setminus \{0\}$, define 
$$
c(\alpha,f):=\inf \left\{a \ : \ j_a^*\alpha\neq 0\in H^*(N^a)\right\}.
$$

\noindent {\bf Exercise.} Prove that $c(\alpha,f)$ is a critical value of $f$ (the proof mimicks that of Proposition~\ref{prop:minimax} above).

The real number $c(\alpha,f)$ is commonly referred to as \emph{the spectral value of $f$ relative to the class $\alpha$}. 

\begin{theorem}[Lyusternik-Schnirelmann 1934] Given $f:N\to\R$ as above, let $\alpha,\beta\in H^*(N)\setminus \{0\}$ be such that 
$$
\alpha=\beta\cup \gamma
$$ 
for some $\gamma$ of degree $\ge 1$. Then 
$$
c(\beta,f)\le c(\alpha,f),
$$ 
and inequality is strict if $f$ has isolated critical points. (The classes $\alpha$ and $\beta$ are called \emph{subordinated}.)
\end{theorem}

\begin{proof}
The inequality $c(\beta,f)\le c(\alpha,f)$ is obvious from the definition: if $\alpha$ is nonzero in $N^a$, then $\beta$ must also be nonzero in $N^a$. 

Assume now that $f$ has isolated critical points. This implies that critical values are isolated and there exist only finitely many critical points on each critical level. Let $\eps>0$ be such that $[c-\eps,c+\eps]$ contains $c=c(\alpha,f)$ as a unique critical value. We claim that $\beta$ is nonzero in $N^{c-\eps}$. Assuming the opposite, we infer from the exact sequence $H^*(N,N^{c-\eps})\to H^*(N)\to H^*(N^{c-\eps})$ that $\beta$ has a representative coming from $H^*(N,N^{c-\eps})$. On the other hand, denoting by $W$ a thickening of the unstable manifolds of the critical points $p_i$ on the level $f^{-1}(c)$, we have that $W$ is a finite union of contractible sets (One cannot guarantee that these sets are discs, unless $f$ is Morse, yet they are clearly contractible onto the points $p_i$). Since the degree of $\gamma$ is $\ge 1$, we have that $\gamma=0$ in $H^*(W)$, so that it has a representative coming from $H^*(N,W)$. Thus $\beta\cup\gamma$ has a representative coming from $H^*(N,N^{c-\eps}\cup W)$. The final observation is that $N^{c-\eps}\cup W$ is a retract of $N^{c+\eps}$, which implies that $\alpha=\beta\cup\gamma$ is zero in $H^*(N^{c+\eps})$, a contradiction.   
\end{proof}

\noindent {\bf Remark.} The above proof makes use of the fact that, upon crossing an isolated critical value $c$, the topology of the sublevel set $N^{c+\eps}$ retracts onto the union of $N^{c-\eps}$ with some collection of disjoint contractible sets obtained by thickening the unstable manifolds of the critical points on the level $c$. This is a generalization of~\eqref{item:crossing} in Theorem~\ref{thm:Morse}.

\section{The Gromoll-Meyer theorem and Bott's iteration formulas} \label{sec:GM} 

In this section we present the ideas involved in the proof of the celebrated theorem of Gromoll and Meyer~\cite{Gromoll-Meyer-JDG1969}.
Berger~\cite{Berger} wrote a clear account of it which can serve as a guide to the beautiful original paper.

\begin{theorem}[Gromoll-Meyer~\cite{Gromoll-Meyer-JDG1969}] \label{thm:GM}
Let $M$ be a simply connected closed manifold.
If the sequence $\big\{b_k(\Lambda M;\mathbb{F}_p)\big\}_{k\ge 0}$ is unbounded for some prime $p\ge 2$, then any Riemannian metric on $M$ admits infinitely many prime closed geodesics. 
\end{theorem}

\noindent \emph{Comments on the assumptions.} 1. Vigu\'e-Poirrier and Sullivan~\cite{Vigue-Poirrier-Sullivan}  proved that, for a simply connected closed manifold $M$, the sequence of rational Betti numbers $b_k(\Lambda M;\mathbb{Q})=\mathrm{rk}\, H^k(\Lambda M;\mathbb{Q})$ is unbounded if and only if $H^*(M;\mathbb{Q})$ needs at least two generators as a ring, i.e. $H^*(M;\mathbb{Q})$ is not isomorphic to a truncated polynomial ring. 
Typical examples of closed manifolds whose cohomology ring with rational coefficients is generated by a single element are compact globally symmetric spaces of rank $1$, but there are also some simply connected compact globally symmetric spaces of rank $>1$ that satisfy this property (see Ziller~\cite[p.~17]{Ziller}). Ziller proved that all the compact globally symmetric spaces of rank $>1$ are such that the sequence of Betti numbers of their free loop space is unbounded with coefficients in $\mathbb{F}_2$. 

2. A fundamental aspect of the Gromoll-Meyer theorem is that it treats arbitrary Riemannian metrics, not only the generic case of metrics all of whose closed geodesics are non-degenerate (\emph{bumpy metrics}).

There are two obstacles for the proof. The first obstacle is that the iterates $\gamma_m:S^1\to M$, $\gamma_m(t)=\gamma(mt)$, $m\in \Z$ of a closed geodesic $\gamma$ are again closed geodesics, and hence contribute to the homology of the free loop space $\Lambda M$. The second obstacle is that we want to deal with degenerate closed geodesics, for which Morse-Bott theory does not apply as such. 

Gromoll and Meyer succeed in overcoming both obstacles by showing in~\cite{Gromoll-Meyer-JDG1969,Gromoll-Meyer-Topology1969} that a (degenerate) closed geodesic contributes to the homology of the free loop space by at most $\nu(\gamma)$ in a range of degrees contained in the interval $[\lambda(\gamma),\lambda(\gamma)+\nu(\gamma)]$. The theorem then follows from the following two statements, which enable one to control the behavior of the index and nullity under iteration. 

\begin{lemma}[Index iteration~\cite{Gromoll-Meyer-JDG1969}] \label{lem:1}
Either $\lambda(\gamma_m)=0$ for all $m\ge 1$ or there exist $\varepsilon>0$ and $C>0$ such that for all $m,s\ge 1$ we have 
$$
\lambda(\gamma_{m+s})-\lambda(\gamma_m)\ge \varepsilon s - C.
$$
\end{lemma}
\begin{lemma}[Nullity iteration~\cite{Gromoll-Meyer-JDG1969}] \label{lem:2}
Either $\nu(\gamma_m)=1$ for all $m\ge 1$ or there exist finitely many iterates $\gamma_{m_1},\dots\gamma_{m_s}$ of $\gamma$ such that for all $m\ge 1$ there exist $i$ and $k$ such that $m=km_i$ and $\nu(\gamma_m)=\nu(\gamma_{m_i})$. In particular the sequence $\nu(\gamma_m)$, $m\ge 1$ takes only finitely many values. 
\end{lemma}

These two lemmas follow from Bott's iteration formulas which we explain below. Granting these and the already mentioned fact that a closed geodesic contributes to the homology of the free loop space by at most $\nu(\gamma)$ in a range of degrees contained in the interval $[\lambda(\gamma),\lambda(\gamma)+\nu(\gamma)]$, one shows by a combinatorial argument that a bound on the number of simple closed geodesics implies a bound on the Betti numbers of $\Lambda M$. This contradiction proves Theorem~\ref{thm:GM}.

In the rest of this section we describe Bott's celebrated iteration of the index method~\cite{Bott-iteration}. Denote 
$$
A_\gamma:H^2(\gamma^*TM)\to L^2(\gamma^*TM), \qquad \xi\longmapsto -\xi''-R(\dot\gamma,\xi)\dot\gamma 
$$
the \emph{asymptotic operator} at a closed geodesic $\gamma$, so that $d^2E(\gamma)(\xi,\eta)=\langle A_\gamma\xi,\eta\rangle_{L^2}$. This is a self-adjoint elliptic operator, hence Fredholm of index $0$, and it has compact resolvent. Its spectrum is therefore discrete with no accumulation point except $\pm\infty$~\cite[p.~187]{Kato}. Moreover, by elliptic regularity the eigenvectors of $A_\gamma$ are smooth. 

Since the curvature tensor $R$ is bounded, the equation $-\xi''-R(\dot\gamma,\xi)\dot\gamma=-\lambda \xi$ has no nontrivial \emph{periodic} solution for $\lambda\gg 0$. Thus the spectrum of $A_\gamma$ is bounded from below (this reproves the finiteness of the index of $\gamma$). 
The index of $\gamma$ is 
$$
\lambda(\gamma)=\sum_{\mathrm{Spec}(A_\gamma)\ni\lambda<0}\mathrm{mult}(\lambda) <\infty,
$$
with $\mathrm{mult}(\lambda)$ the multiplicity of the eigenvalue $\lambda$. The nullity of $\gamma$ is 
$$
\nu(\gamma)=\dim\ker A_\gamma.
$$
We have $1\le \nu(\gamma)\le 2n-1$. The first inequality follows from the fact that $\dot\gamma\in\ker\,  A_\gamma$. The second inequality follows as in~\S\ref{sec:Morse} since an element in the kernel of $A_\gamma$ solves a second order ODE. 

Denote by $\mathcal{V}_\gamma$ the space of smooth vector fields along $\gamma$, not necessarily periodic, and define the operator 
$$
\mathcal{L}:\mathcal{V}_\gamma\to\mathcal{V}_\gamma,\qquad \xi\longmapsto -\xi''-R(\dot\gamma,\xi)\dot\gamma.
$$
Given an integer $m\ge 1$ and a complex number $\mu\in\C$, we consider the following two boundary value problems.
\renewcommand{\theenumi}{\roman{enumi}}
\begin{enumerate}
\item $$ \mathcal{L}\xi=\mu\xi,\qquad \xi(t+m)=\xi(t), \qquad \xi'(t+m)=\xi'(t).$$
We denote $\Theta_m(\mu)$ the dimension of the space of solutions, so that 
$$
\lambda(\gamma_m)=\sum_{\mu<0}\Theta_m(\mu),\qquad \nu(\gamma_m)=\Theta_m(0).
$$
\item We complexify the operator $\mathcal{L}:\mathcal{V}_\gamma\to\mathcal{V}_\gamma$ to $L:V_\gamma=\mathcal{V}_\gamma\otimes_{\mathbb{R}}\mathbb{C}\to V_\gamma$ via the formula $L(\xi+i\eta):=\mathcal{L}\xi +i\mathcal{L}\eta$, and consider for $z\in S^1$ the problem 
$$
LY=\mu Y, \qquad Y(t+m)=zY(t), \qquad Y'(t+m)=zY'(m).
$$
We denote $\Theta_m^z(\mu)$ the dimension of the space of solutions. The fact that $\dot\gamma\in\ker\, A_\gamma$ translates into $\Theta_m^1(0)\ge 1$ for all $m\ge 1$.
\end{enumerate}
The beautiful observation of Bott~\cite{Bott-iteration} is that 
\begin{equation}\label{eq:Bott}
\Theta_m^z(\mu)=\sum_{w^m=z}\Theta_1^w(\mu),
\end{equation}
and in particular $\Theta_m(\mu)=\sum_{w^m=1}\Theta_1^w(\mu)$. Equation~\eqref{eq:Bott} follows directly from the fact that our second order ODE operator $\cL$ is time-independent. As a consequence of~\eqref{eq:Bott}, by setting 
 $$
\Lambda(z):=\sum_{\mu<0}\Theta_1^z(\mu), 
$$
$$
N(z):=\Theta_1^z(0)  \ \mbox{ for } \ z\neq 1,
$$
and
$$
N(1):= \Theta_1^1(0)-1,
$$
we have 
\begin{equation}\label{eq:LN}
\lambda(\gamma_m)=\sum_{z^m=1}\Lambda(z), \qquad \nu(\gamma_m)=1+\sum_{z^m=1}N(z).
\end{equation}

In order to understand the behavior of the index and of the nullity of a geodesic $\gamma$ under iteration one is therefore reduced to studying the two functions $\Lambda:S^1\to\N$ and $N:S^1\to\N$ associated to $\gamma$. Bott understood them and discovered that they are well-behaved. 

\begin{theorem}[properties of $\Lambda$ and $N$, cf. Bott~\cite{Bott-iteration}] \label{thm:Bott}
Let $\dim M=n$. 
\begin{enumerate}
\item $\Lambda(z)=\Lambda(\bar z)$, $N(z)=N(\bar z)$. 
\item \label{item:Bott-2} $N(z)=0$ except for at most $2n-2$ points, called \emph{Poincar\'e points}, which are the eigenvalues lying on the unit circle of a symplectic matrix of size $2n-2$ (the linearized Poincar\'e return map, cf. below).
\item \label{item:Bott-3} $\Lambda$ is locally constant except possibly at the Poincar\'e points. 
\item For any $z_0\in S^1$, the \emph{splitting numbers} $S^\pm(z_0)=\lim_{z\to z_0^\pm}\Lambda(z)-\Lambda(z_0)$ are nonnegative and bounded by $N(z_0)$.  In particular $\Lambda$ is lower semi-continuous. 
\end{enumerate}
\end{theorem}

The proof of Theorem~\ref{thm:Bott} uses in a crucial way the linearized Poincar\'e return map $P$ for the geodesic flow along the geodesic $\gamma$, which we now describe following~\cite{BTZ1982} (see also~\cite[Chapter~3]{Klingenberg-book}). Given a vector $v\in TM$, denote $\gamma_v$ the unique geodesic such that $\dot\gamma_v(0)=v$. The geodesic flow $\Phi_t:TM\to TM$ is given by $\Phi_tv:=\dot\gamma_v(t)$. Thus $\Phi_t$ maps $S_\rho TM:=\{v\inÊTM\, : \, \|v\|=\rho\}$ to itself. If $\gamma:S^1=\R/\Z\to M$ is a closed geodesic with $v=\dot\gamma(0)=\dot\gamma(1)$, then $\{\Phi_t v\}_{t\in [0,1]}$ is a periodic orbit of $\Phi$. The Poincar\'e map $\mathcal P$ along $\gamma$ is defined to be the return map of a local hypersurface $\Sigma\subset S_\rho TM$ transverse to the orbit $\Phi_tv$. The linearization $P:=D_v\mathcal P$ is independent of $\Sigma$ up to conjugacy and is called \emph{the linearized Poincar\'e map along $\gamma$}. We can choose $\Sigma$ such that $T_v\Sigma$ is $E\oplus E$, where $E$ is the orthogonal complement of $v$ in $T_pM$, $p=\gamma(0)$. Here $T_vTM$ is identified with $T_pM\oplus T_p M$ via the decomposition of $TTM$ into horizontal and vertical subspaces. The linearized Poincar\'e map $P:E\oplus E\to E\oplus E$ is then given by $(A,B)\longmapsto (Y(1),Y'(1))$, where $Y$ is the unique Jacobi field along $\gamma$ with initial conditions $(Y(0),Y'(0))=(A,B)$, solving $LY=0$. The linearized Poincar\'e map preserves the symplectic structure $\omega$ on $E\oplus E$ defined by $\omega((A_1,B_1),(A_2,B_2))=\langle A_1,B_2\rangle-\langle B_1,A_2\rangle$.

Bott's theorem is thus of a symplectic nature and, as such, it has been generalized to periodic orbits of autonomous Hamiltonian systems by Long~\cite{Long}. The counterpart of the Morse index is then the Maslov index. That the two coincide in the case of the geodesic flow was proved by Viterbo~\cite[Theorem~3.1]{Viterbo1990}. 

\begin{example} \label{example:hyperbolic}
Assume that $\gamma$ is a simple hyperbolic closed geodesic, meaning that the eigenvalues of the linearized Poincar\'e return map lie outside the unit circle. Then the function $N$ vanishes identically (Theorem~\ref{thm:Bott}(\ref{item:Bott-2})), the function $\Lambda$ is constant (Theorem~\ref{thm:Bott}(\ref{item:Bott-3})), and therefore $\lambda(\gamma_m)=\sum_{z^m=1}\Lambda(z)=m\Lambda(1)=m\lambda(\gamma)$. We have already encountered hyperbolic closed geodesics in Remark~\ref{rmk:invisible}.
\end{example}

\smallskip

\noindent \emph{Comments on Theorem~\ref{thm:Bott}.} Assertion (i) holds since $L$ is the complexification of a real operator. To prove (ii) one uses that $N(z)=\dim_{\mathbb{C}}\ker\, (P-z\mathrm{Id})$, i.e. the Poincar\'e points are the eigenvalues of $P$ of absolute value equal to $1$. In particular, there are at most $2n-2$ of them. (iii) is to be interpreted as the fact that $\Lambda$ can change only if an eigenvalue crosses $0$. (iv) describes a general property of functions of geometric origin. 

\smallskip

Bott's theorem, together with~\eqref{eq:LN}, implies by a straightforward computation the statements of Lemmas~\ref{lem:1} and~\ref{lem:2}~\cite{Gromoll-Meyer-JDG1969}.

\section{Infinitely many closed geodesics for generic metrics} \label{sec:equivariant}

As already mentioned in the introduction, it is still an open question whether any Riemannian metric on a given closed simply connected manifold $M$ admits infinitely many geometrically distinct prime closed geodesics. A summary of the current state-of-the-art as understood by the author is the following. 
\begin{itemize}
\item (cf.~\S\ref{sec:GM}) if the sequence of Betti numbers $b_k(\Lambda M;\F)$ with coefficients in some field $\F$ is unbounded, the answer is affirmative (Gromoll-Meyer). This is in particular true for manifolds whose rational cohomology ring requires at least two generators (Vigu\'e-Poirrier-Sullivan), and also for globally symmetric spaces of rank $>1$ (Ziller). 
\item (cf. the discussion in Appendix~\ref{app:S2} below) among the spaces which do not satisfy the assumptions of the Gromoll-Meyer theorem, the problem is solved in the affirmative for the $2$-sphere (Bangert, Franks, Angenent, Hingston). The solution uses in an essential way tools from $2$-dimensional dynamical systems. 
\item the problem is solved in the affirmative for a \emph{generic} metric (Klingenberg, Takens, Hingston, Rademacher), and we give in this section an overview of the relevant ideas.    
\end{itemize}

It is tantalizing that the problem of the existence of infinitely many closed geodesics on the spheres $S^n$, $n\ge 3$ remains open in full generality. This was the original problem that Morse addressed in his foundational essay~\cite{Morse}!

We now give an overview of the ideas involved in the proof of the following result, which is the combination of work of several mathematicians. 

\begin{theorem}[Klingenberg, Takens, Hingston, Rademacher, ...] \label{thm:generic}
Let $M$ be a simply connected smooth closed manifold. For a $C^4$-generic Riemannian metric there are infinitely many geometrically distinct prime closed geodesics on $M$. 
\end{theorem}

\noindent \emph{Note on the genericity assumption.} The genericity assumption allows to assume that all closed geodesics are non-degenerate. Klingenberg and Takens~\cite{Klingenberg-Takens},~\cite[Theorem~3.3.10]{Klingenberg-book} have proved that, for a $C^4$-generic metric on a closed manifold $M$, either there exists a closed geodesic \emph{of twist type} (see~\cite[p.~103]{Klingenberg-book} for the definition, in particular all the eigenvalues of the symplectic Poincar\'e return map have absolute value equal to $1$), or all closed geodesics are \emph{hyperbolic} (this means that the eigenvalues of the Poincar\'e return map are situated off the unit circle). In the former case, a smooth version of the Birkhoff-Lewis fixed point theorem due to Moser~\cite{Moser} ensures that there are infinitely many prime closed geodesics in a neighborhood of the geodesic of twist type.
These closed geodesics accumulate onto the latter, with minimal period going to infinity. 

In order to prove the generic existence of infinitely many prime closed geodesics on simply connected manifolds, one is therefore left to deal with the situation in which all closed geodesics are hyperbolic. 
This case was settled by Rademacher. 

\begin{theorem}[Rademacher~{\cite[Theorem~1]{Rademacher1989}}]
If $M$ is a simply connected closed Riemannian manifold whose rational cohomology ring is a truncated polynomial ring in one even variable and such that all its closed geodesics are hyperbolic, then there are infinitely many geometrically distinct ones. 
\end{theorem}

A previous theorem of Hingston, going in the same direction, is the following. 

\begin{theorem}[Hingston~{\cite[Theorem~6.2]{Hingston1984}}]
Let $M$ be a simply connected closed Riemannian manifold whose rational homotopy type is that of a compact rank $1$ symmetric space ($S^n$, $\mathbb{C}P^n$, $\mathbb{H}P^n$, or $CaP^2$ -- the octonion plane). Assume all closed geodesics on $M$ are hyperbolic. The number $n(\ell)$ of prime closed geodesics of length $\le \ell$ grows at least as fast as the prime numbers: $\liminf_\ell n(\ell)\frac {\log(\ell)} \ell >0$. In particular, there are infinitely many prime closed geodesics. 
\end{theorem}

Hingston is able to go beyond the Gromoll-Meyer theorem because she uses $S^1$-equivariant homology of $\Lambda M$. Rademacher is able to go beyond the results of Hingston because, in addition to using $S^1$-equivariant homology, he proves a resonance identity involving the mean indices of the closed geodesics. Remarkably enough, the same kind of resonance identity appeared simultaneously in relation with the search for closed characteristics on convex hypersurfaces in $\R^{2n}$~\cite{Viterbo} (see also the discussion in~\S\ref{sec:symplectic} below). 

Since the cyclic groups $\Z_m$, $m\ge 2$ are $\Q$-acyclic, we have isomorphisms  $H^*_G(\Lambda M,M;\Q)\simeq H^*(\Lambda M/G,M;\Q)$, where $H^*_G(\cdot)$ denotes $G$-equivariant cohomology for $G=\Z_m$. Rademacher~\cite{Rademacher1989} uses only rational coefficients and works therefore with the usual quotient $(\Lambda M/G,M)$. In contrast Hingston~\cite{Hingston1984} makes essential use of $\F_p$-coefficients (for varying $p$) and works with the homotopy quotient. One notable consequence of the assumption that all orbits are hyperbolic is the identity $\lambda(\gamma_m)=m\lambda(\gamma)$ (see Example~\ref{example:hyperbolic}).

Rather than giving details for the proofs of the above results, in the remainder of this section we will content ourselves to motivate the use of equivariant homology. Bott describes in~\cite[Lecture~4, p.~350]{Bott-oldnew} the following phenomenon arising on the $n$-sphere, endowed with the round metric. Recall that the closed geodesics are great circles covered $m$ times and as such they form critical manifolds $V_m$, $m\ge 1$ of dimension $2n-1$, diffeomorphic to the unit tangent bundle $T_1S^n$. Indeed, every closed geodesic is uniquely determined by its origin and the tangent vector at the origin. Now:   
\begin{itemize}
\item If one deforms the sphere into an ellipsoid $\{\sum_{i=1}^{n+1}a_ix_i^2=1\}\subset\R^{n+1}$ with $a_1<a_2\dots<a_{n+1}$, the first critical manifold $V_1$ decomposes into the $n(n+1)/2$ geodesics given by the intersection of the coordinate planes with the ellipsoid. 
\item The homology of $V_1\simeq T_1S^n$ is $4$-dimensional, supported in degrees $0$, $n-1$, $n$, and $2n-1$. Moreover, $V_1$ admits a perfect Morse function $f$ with exactly $4$ points (viewing $T_1S^n\subset \R^{n+1}\times \R^{n+1}\ni(x,v)$, we can take $f(x,v)=\|x\|^2+\|v\|^2$). Thus, according to classical Morse theory, under small perturbations $V_1$ should contribute no more than $4$ critical points! 
\end{itemize}
We quote from Bott~\cite[p.~350]{Bott-oldnew}: 

\begin{center}
\emph{"the correct diagnosis of this ailment is that our energy function has a built-in symmetry which has to be taken into account before the proper correspondence between geometry and topology is realized".}
\end{center}
 
\noindent {\it Remark.} A familiar situation is that of odd-dimensional spheres $S^{2n+1}\subset \C^{n+1}$, carrying the $S^1$-action $t\cdot(z_1,\dots,z_n)=(e^{2\pi it}z_1,\dots,e^{2\pi it} z_n)$. Although the homology of $S^{2n+1}$ is $2$-dimensional, a generic $S^1$-invariant function on $S^{2n+1}$ has at least $n+1$ critical orbits, with $n+1=\mathrm{rk} \, H_*(\mathbb{C}P^n)$. Here, of course, $\mathbb{C}P^n=S^{2n+1}/S^1$.

In the case of geodesics on $S^n$ that we are discussing, the symmetry group $O(2)$ acts freely on $V_1$ and the quotient is $G(2,n+1)$, the Grassmannian of $2$-planes in $\R^{n+1}$, whose total homology rank is precisely $n(n+1)/2$. 

On the other hand, $O(2)$ acts on $V_m$, $m\ge 2$ with isotropy group $\Z_m$. The na\"\i ve quotient is still $G(2,n+1)$, but there is a much better way to perform the quotient operation in order to take into account the structure of the isotropy groups, namely by considering the homotopy quotient. Given a group $G$, there is up to homotopy a unique contractible space on which $G$ acts freely, denoted $EG$. The quotient $BG=EG/G$ is called the \emph{classifying space of $G$}, whereas the $G$-principal bundle $EG\to BG$ is a universal $G$-principal bundle. The \emph{homotopy quotient} of a $G$-space $X$ is by definition 
$$
X_G:=X \times_{G} EG,
$$ 
and the \emph{$G$-equivariant homology groups of $X$} are defined by
$$
H^G_*(X):=H_*(X_G).
$$
In the case $G=SO(2)$ one can take as a model for $EG$ the inductive limit $\lim_{n\to\infty} S^{2n-1}$, whereas in the case $G=O(2)$ one can take as a model for $EG$ the inductive limit $\lim_{n\to\infty} T_1S^n$. 

The classifying space $BG=EG/G$ captures subtle information on the group $G$. As an example, let us consider the following fact: $O(2)$ is \emph{not} a direct product of $SO(2)=S^1$ by $\Z/2$. However, it is very close to it: $O(2)$ is the semi-direct product $SO(2)\rtimes \{\pm 1\}$ with respect to the nontrivial action of $\{\pm 1\}$ on $SO(2)= U(1)=\{z\in\mathbb{C}\, : \, |z|=1\}$ given by $(-1)\cdot z := \bar z$. Alternatively, we can view $O(2)$ as being the extension  
$$
1\to SO(2)\to O(2)\to \{\pm 1\}\to 1.
$$
From this point of view, the action of $\{\pm 1\}$ on $SO(2)$ is induced by the non-canonical splitting $-1\mapsto \mathrm{Diag}(1,-1)$.  

That $O(2)\not\simeq SO(2)\times \mathbb{Z}/2$ is reflected by the fact that $BO(2)\not\simeq BSO(2)\times B\mathbb{Z}/2$. Indeed, it is proved in~\cite[Theorem~3.16]{HatcherVBKT} that the torsion free part of the cohomology of $BO(2)=G(2)$, the Grassmannian of $2$-planes in $\mathbb{R}^\infty$, is a polynomial ring in one variable $p_1$ of degree $4$. On the other hand, it follows from the K\"unneth formula that the torsion free part of $BSO(2)\times B\mathbb{Z}/2=\mathbb{C}P^\infty\times\mathbb{R}P^\infty$ is a polynomial ring in one variable $c_1$ of degree $2$.

The following three test cases explain the main features of $G$-equivariant homology: 
\renewcommand{\theenumi}{\roman{enumi}}
\begin{enumerate}
\item if $G$ acts freely, then $H_*^G(X)=H_*(X/G)$; 
\item $H_*^G(pt)=H_*(BG)$; more generally, if $G$ acts trivially on $X$ then $H_*^G(X)=H_*(X\times BG)$; 
\item if $H$ is a closed subgroup of $G$, then $H_*^G(G/H)=H_*(BH)$. 
\end{enumerate} 

The upshot is that, from a Morse theoretic point of view, the $m$-fold iterate $\gamma_m$ of a closed geodesic contributes to $H^{O(2)}_*(\Lambda M)$ by $H_*(B\Z_m)$. Now~\cite[\S3]{Brown} 
$$
H_*(B\Z_m;\Z)=\left\{\begin{array}{ll} \Z, & *=0, \\Ê\Z_m, & * \mbox{ odd},\\ 0,& *>0 \mbox{ even}.
\end{array}\right.Ê
$$
Using the universal coefficient theorem we therefore obtain 
$$
H_*(B\Z_m;\Z_\ell)=\left\{\begin{array}{ll} \Z_\ell, & *=0, \\Ê\Z_d, & *>0,
\end{array}\right.
$$
with $d$ the greatest common divisor of $m$ and $\ell$, and also
$$
H_*(B\Z_m;\Q)=\left\{\begin{array}{ll}\Q, & *=0, \\Ê0, & *>0.
\end{array}\right.
$$

Going back to our $O(2)$-manifolds $V_m$, $m\ge 2$ with isotropy group $\Z_m$, we see that $H_*^{O(2)}(V_m)=H_*(G(2,n+1)\times B\Z_m)$. Over the rationals the homology is that of $G(2,n+1)$, but over $\Z_\ell$ we see $\Z_d$-torsion appearing. According to Bott~\cite[p.~354]{Bott-oldnew}, 

\begin{center}
\emph{``in one way or another it was this torsion which was improperly accounted for in Morse's attempts on the question [of closed geodesics] on $S^n$ long ago".}
\end{center}

\begin{remark}[$SO(2)$ vs. $O(2)$] \label{rmk:SO2-O2} It is puzzling that the results of Rademacher and Hingston only make use of $SO(2)$-equivariant homology, and not of $O(2)$-equivariant homology. The $O(2)$-symmetry must play a distinctive role in the problem of existence of closed geodesics for Riemannian metrics, if ever the existence of infinitely many closed geodesics is to be true in full generality. Indeed, we discuss in~\S\ref{sec:symplectic} below an example due to Katok of a non-symmetric Finsler metric on the $n$-sphere, $n\ge 2$ which has only a finite number of geometrically distinct prime closed geodesics. The natural symmetry group of a non-symmetric Finsler metric is precisely $SO(2)$, and \emph{not} $O(2)$. Of course, this does not contradict Theorem~\ref{thm:generic} which applies to generic Riemannian metrics. The analogous existence result of infinitely many closed geodesics does also hold for $C^2$-generic Finsler metrics, although by an entirely different, Hamiltonian, proof~\cite[p.~141]{Ziller-Katok}. 
\end{remark}

\section{Computation of the homology of some classical free loop spaces using Morse theory} \label{sec:completing}

In the previous sections we used Morse theory in order to infer statements on closed geodesics from homological properties of the free loop space (the topology constrains the geometry). In this section we adopt the opposite perspective: in some situations the geometry of the underlying manifold is so explicit that we have a complete and detailed knowledge of the geodesic flow, and this allows us to compute the homology of free loop spaces (the geometry determines the topology). This method, which goes back to Bott's proof of the periodicity theorem~\cite{Bott-Stable_homotopy} and to the work of Bott and Samelson~\cite{Bott-Samelson}, was successfully applied by Ziller~\cite{Ziller}.

Let us call a Morse-Bott function on a Hilbert manifold $X$ \emph{perfect} if all the boundary maps in all the long exact sequences~\eqref{eq:lespair} of all the pairs $(X^{\le c},X^{<c})$ vanish. Ziller~\cite{Ziller} proved that the energy functional on the free loop spaces of compact rank one symmetric spaces $S^n$, $\C P^n$, $\H P^n$, $Ca P^2$ is perfect for any choice of coefficients, whereas in the case of $\R P^n$ the energy functional is perfect with $\Z/2$-coefficients. We describe in this section the relevant geometric objects in order to understand this computation in the case of $\Lambda S^n$ and $\Lambda \C P^n$. 

The key fact in Ziller's proof is that the relative cycles provided by Morse theory which describe the change in topology of the sub-level sets upon crossing a critical value can be ``completed'' inside the sub-level set. The resulting notion of a ``completing manifold'' is reminiscent of~\cite[p.~531]{Atiyah-Bott} and~\cite[p.~97]{Hingston1984} (see also~\cite[IX.7]{Morse}, \cite[p.~979]{Bott-Samelson}). A detailed analysis of this notion is carried out in~\cite{Hingston-Oancea}.

\begin{definition}[\cite{Hingston-Oancea}]
\label{def:completing} Let $X$ be a Hilbert manifold and $f:X\to\mathbb{R}$ a
$C^{2}$-function satisfying condition~(C) of Palais and Smale as stated in Theorem~\ref{thm:PS}. 
Let
\[
K:=\mathrm{Crit}(f)\cap f^{-1}(0)\subset X
\]
be the critical locus of $f$ at level $0$ and assume $K$ is a Morse-Bott
non-degenerate closed manifold of index $\iota(K)$. 

A \emph{completing manifold for $K$} is a finite dimensional closed manifold $Y$
together with a closed submanifold $L\subset Y$ of codimension $\iota(K)$ and a map
$\varphi:Y\rightarrow X^{\leq0}$ subject to conditions~\eqref{item:embedding}
and~\eqref{item:surjective}  below. We say that $Y$ is a \emph{strong
completing manifold} if it satisfies conditions~\eqref{item:embedding}
and~\eqref{item:retraction} (in which case condition~\eqref{item:surjective}
follows).

\begin{enumerate}
\item \label{item:embedding} the map $\varphi$ is an embedding \emph{near}
$L$, it maps $L$ diffeomorphically onto $K$, and
\[
\varphi^{-1}(K)=L.
\]

\item \label{item:surjective} the canonical map
\[
H_{\cdot}(Y)\to H_{\cdot}(Y,Y\setminus L)
\]
is surjective for any choice of coefficient ring.

\item \label{item:retraction} the embedding $s:L\hookrightarrow Y$ admits a retraction $p:Y\to L$ such that $ps=\operatorname{Id}_L$ and the manifold $Y$ is orientable. \end{enumerate}
\end{definition}

The definition can be refined by allowing~\eqref{item:surjective} to hold only for certain coefficients. If we use $\Z/2$-coefficients, the orientability assumption in~\eqref{item:retraction} can be dropped. However, we will not need such refinements in the sequel. 

It is straightforward to prove that (iii) implies (ii) (see~\cite{Hingston-Oancea}). The proof uses the shriek map $p_!:H_{\cdot- \mathrm{codim}(L)}(L;o_\nu)\to H_\cdot(Y)$, where $o_\nu$ is the orientation local system of the normal bundle to $L$ in $Y$.

\begin{lemma}[\cite{Hingston-Oancea}]
\label{lem:completing} Let $X$, $f$, and $K$ be as in
Definition~\ref{def:completing} and denote by $\nu^{-}$ the negative bundle of
$K$ (of rank $\iota(K)$). If $K$ admits a completing manifold then we have
short exact sequences
\[
0\rightarrow H_{\cdot}(X^{<0})\rightarrow H_{\cdot}(X^{\leq0})\rightarrow
H_{\cdot-\iota(K)}(K;o_{\nu^{-}})\rightarrow0.
\]
If $K$ admits a strong completing manifold these exact sequences are split, so
that
\[
H_{\cdot}(X^{\leq0})\simeq H_{\cdot}(X^{<0})\oplus H_{\cdot-\iota(K)}
(K;o_{\nu^{-}}).
\]
\end{lemma}

\begin{proof}
By an arbitrarily small perturbation of
$\varphi$ along the negative gradient flow of $-\nabla f$ we will have pushed all noncritical points at
level $0$ below level $0$, so that we obtain a map
$(Y,Y\setminus L)\rightarrow(X^{\leq0},X^{<0})$
that satisfies the same conditions as $\varphi$, and which we shall still denote by $\varphi$. By functoriality of the long exact sequence of a pair we get a commutative diagram of long exact sequences 
$$
\xymatrix{
\dotsÊ\ar[r] & H_\cdot(X^{<0}) \ar[r] & H_\cdot(X^{\le 0}) \ar[r]^-{i_*} & H_\cdot(X^{\le 0},X^{<0}) \ar[r]^-{[-1]} & \dots \\
\dots \ar[r] & H_\cdot(Y\setminus L) \ar[r] \ar[u]Ê& H_\cdot(Y) \ar[r]_-{j_*}Ê\ar[u]Ê& H_\cdot(Y,Y\setminus L) \ar[r]_-{[-1]} \ar[u]_\simeq & \dots
}
$$
Our assumptions imply that the normal bundle to $L$ in $Y$ is isomorphic to the pull-back via $\varphi$ of the negative bundle over $K$. Using excision, the Thom isomorphism, and Theorem~\ref{thm:Morse}.\eqref{item:crossing}, we obtain that the rightmost vertical arrow is an isomorphism. Surjectivity of the map $j_*$ then clearly implies surjectivity of the map $i_*$. In the case of a strong completing manifold, the map $j_*$ has a section and this induces a section of $i_*$ as well. 
\end{proof}

\subsection{Computation of the homology of $\Lambda S^n$}

Let us consider on $S^n$ the Riemannian metric with constant curvature $1$. The critical points of the energy functional $E:\Lambda S^n\to \R_+$ are either the constant loops, on the minimum level $0$, or great circles traversed $k\ge 1$ times, of length $2k\pi$ and energy $(2k\pi)^2$. Since a geodesic is uniquely determined by its starting point and by its initial speed vector, we infer that the critical set $\mathrm{Crit}(E)$ is a disjoint union of submanifolds $K_k$, $k\ge 0$ with 
$$
K_0\simeq S^n \qquad \mbox{and} \qquad K_k\simeq STS^n, \ k\ge 1.
$$
Here $STS^n$ denotes the unit tangent bundle of $S^n$, $K_0$ is the critical manifold of constant loops, and $K_k$, $k\ge 1$ is the critical manifold of great circles traversed $k$ times. The absolute minimum critical level set has index 
$$
\iota(K_0)=0
$$
and nullity 
$$
\nu(K_0)=n.
$$
Ziller~\cite{Ziller} has computed the index and nullity of the critical manifolds $K_k$, $k\ge 1$ and the answer is (see also Example~\ref{exple:indexSn})
$$
\iota(K_k)= (2k-1)(n-1),\qquad \nu(K_k)=2n-1.
$$
In particular, since $\nu(K_k)=\dim\, K_k$ for all $k\ge 0$ we infer that the $K_k$'s are Morse-Bott non-degenerate critical submanifolds. 

Let us now describe a completing manifold for $K_1$. Denote 
$$
S^{n-1}\hookrightarrow Y_1 \to STS^n
$$ 
the following $S^{n-1}$-bundle: given a point $x\in S^n$ and a unit vector $v\in T_xS^n$, the fiber $S^{n-1}_{x,v}$ over $(x,v)$ is the $(n-1)$-dimensional equator of $S^n$ that is orthogonal to the great circle passing through $x$ and tangent to $v$. Note that the vector $v$ singles out one of the two half-spheres with boundary $S^{n-1}_{x,v}$, namely the one half-sphere such that $v$ points towards its interior. 

Denote $L_1:=STS^n$ and $s:L_1\hookrightarrow Y_1$ the canonical section defined by associating to $(x,v)$ the antipodal point $x^*$ on the equator $S^{n-1}_{x,v}$. Define 
$$
\varphi:Y_1\to \Lambda S^n
$$
by associating to a triple $(x,v,y)$ with $x\in S^n$, $v\in ST_xS^n$, $y\in S^{n-1}_{x,v}$ the unique circle on $S^n$ which passes through $x$ and $y$, which is tangent to $v$, and which is orthogonal to $S^{n-1}_{x,v}$. This vertical circle has to be understood as being constant if $y=x$. If non-constant, it is oriented by $v$ and as such it admits a unique parametrization proportional to arc-length on the interval $[0,1]$. 

\begin{figure}
\begin{center}
\input{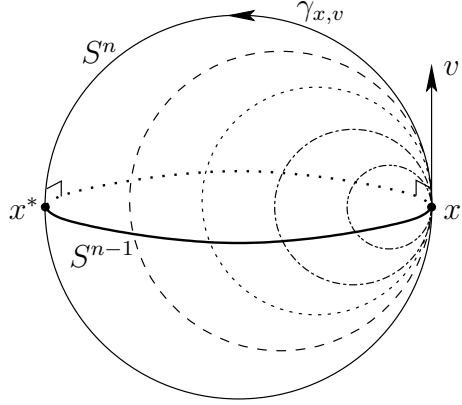}
\end{center}
\caption{Completing manifold for $K_1$ in $\Lambda S^n$.}
\label{fig:Y1Sn}
\end{figure}

One can check directly that condition (i) in the definition of a completing manifold is satisfied. In particular, one notices that the dimension of $Y_1$ is equal to $\iota(K_1)+\nu(K_1)$. That condition (iii) is also satisfied follows again from the construction, since $Y_1$ is naturally a bundle over $L_1$, and the embedding $s:L_1\hookrightarrow Y_1$ is a section of this fibre map. Moreover, $Y_1$ is orientable.

The completing manifold $Y_1$ is the building block for constructing completing manifolds $Y_k$ for $K_k$, $k\ge 2$. Indeed, we set 
$$
Y_k:= Y_1 \ {_{\mathrm{ev}}}\!\!\times_{\mathrm{ev}} Y_1 \ {_{\mathrm{ev}}}\!\!\times \dots \times_{\mathrm{ev}} Y_1 \qquad (k\ \mbox{factors})
$$
where $\mathrm{ev}:Y_1\to S^n$ is the evaluation map at the origin composed with the representation map $\varphi:Y_1\to \Lambda S^n$. We define 
$$
L_k:=\{\big((x,v,x^*), (x,v,x^*),\dots, (x,v,x^*)\big)\, : \, (x,v)\in STS^n\}.
$$
Finally, we define a smooth map $\varphi:Y_k\to \Lambda S^n$ by concatenating the $k$ circles with the same origin that are determined by any point in $Y_k$ via the map $\varphi:Y_1\to \Lambda S^n$. One can thus think of $Y_k$ as being a space of ``strings of vertical $k$-circles with the same origin". 

Again, a direct check shows that $Y_k$ is a completing manifold for $K_k$. Indeed, the map $\varphi$ maps $L_k$ diffeomorphically onto $K_k$ and the intersection of $\varphi(Y_k)$ with the critical level of $K_k$ is exactly $K_k$. The dimension of $Y_k$ is 
$$
\dim\, Y_k = (2k+1)(n-1)+1,
$$ 
and this is equal to $\iota(K_k)+\nu(K_k)$ (note to this effect that the evaluation map $\mathrm{ev}:Y_1\to S^n$ is a submersion). Thus condition (i) is satisfied. In order to check condition (iii), we define the map $p:Y_k\to STS^n$ as being the projection $Y_1\to STS^n$ composed with the projection $Y_k\to Y_1$ onto the first component. This defines a fiber bundle structure on $Y_k$ for which the embedding $s:L_k\hookrightarrow Y_k$ is a section. Finally, since $Y_k$ is a fiber-product of orientable manifolds over orientable manifolds, it is itself orientable. Thus condition (iii) is satisfied. 

The outcome is that the following isomorphism holds with arbitrary coefficients 
$$
H_\cdot(\Lambda S^n)=H_\cdot(S^n)\oplus \bigoplus _{k\ge 1} H_\cdot(STS^n)[-(2k-1)(n-1)].
$$
An even better way to write the outcome of the computation is in a table, where the horizontal coordinate represents the critical values of the energy functional, and the vertical coordinate represents homological degrees.

\begin{center}
\begin{tabular}
[c]{c|cccccc}
$\vdots$ &  &  &  &  & & \\
$5(n-1)+1$ &  &  &  &  & & \\
$5(n-1)$ &  &  &  &  & \dots & \\
$\vdots$ &  &  &  &  & & \\
$\vdots$ &  &  &  &  & & \\
$3(n-1)+1$ &  &  &  &  & & \\
$3(n-1)$ &  &  &  &  & & \\
$\vdots$ &  &  &  &  & & \\
$\vdots$ &  &  &  &  & & \\
$n$ &  &  &  &  & & \\
$n-1$ &  &  &  &  & & \\
$\vdots$ &  &  &  &  & & \\
$0$ &
\raisebox{0pt}[0pt][0pt]{\raisebox{40pt}{\dbox{$\xymatrix@R=8pt{ \\ H_\cdot(S^n) \\ \quad }$}}} &
\raisebox{0pt}[0pt][0pt]{\raisebox{103pt}{\dbox{$\xymatrix@R=24pt{ \\ H_\cdot(STS^n) \\ \quad }$}}} &
\raisebox{0pt}[0pt][0pt]{\raisebox{165pt}{\dbox{$\xymatrix@R=24pt{ \\ H_\cdot(STS^n) \\ \quad }$}}} &
&
\\\hline
& $0$ & $(2\pi)^2$ & $(4\pi)^2$ & & $(6\pi)^2$ & \dots
\end{tabular}
\end{center}

\subsection{Computation of the homology of $\Lambda \C P^n$}

Let us endow $\C P^n$ with the Fubini-Study metric induced from the metric with constant sectional curvature equal to $1$ on $S^{2n+1}$. Following Besse~\cite[Proposition~3.32]{Besse} the geodesics on $\C P^n$ admit the following description: given $x\in \C P^n$ and $v\in ST_x\C P^n$, there is a unique complex line $\ell_{x,v}$ through $x$ and tangent to $v$. This complex line is a metric sphere of dimension two with constant curvature equal to $4$ and is a totally geodesic submanifold of $\C P^n$. As such, the geodesic $\gamma_{x,v}$ through $x$ in the direction $v$  
is a great circle on $\ell_{x,v}$. In particular, all the geodesics of $\C P^n$ are closed and the nontrivial simple closed geodesics all have length $\pi$, and energy $\pi^2$. 

The critical set of the energy functional decomposes therefore as a disjoint union of submanifolds $K_k$, $k\ge 0$ with 
$$
K_0\simeq \C P^n \qquad \mbox{and} \qquad K_k\simeq ST\C P^n,\ k\ge 1.
$$
The critical manifold $K_0$ consists of constant geodesics and so its index and nullity are 
$$
\iota(K_0)=0,\qquad \nu(K_0)=2n.
$$
The critical manifold $K_k$, $k\ge 1$ consists of closed geodesics of length $k\pi$. The index and nullity have been computed by Ziller~\cite{Ziller} and are given by (see also Example~\ref{exple:indexCPn})
$$
\iota(K_k)=2(k-1)n+1,\qquad \nu(K_k)=4n-1.
$$
In particular, all the critical manifolds of the energy functional are Morse-Bott non-degenerate. 

We now describe a completing manifold for $K_1$. We define $Y_1$ to be the $S^1$-bundle  
$$
S^1\hookrightarrow Y_1 \to ST\C P^n
$$
whose fiber at $(x,v)$ is the equator $e_{x,v}\subset \ell_{x,v}$ that is orthogonal to $v$. Note that this bundle admits a natural trivialization 
$$
Y_1\simeq ST\C P^n\times S^1_\pi,
$$
with $S^1_\pi:=\R/\pi\Z$. Indeed, the equator $e_{x,v}$ is a metric circle of length $\pi$, and it has a natural orientation defined as the boundary orientation of the half-sphere $\ell^+_{x,v}\subset \ell_{x,v}$ determined by the condition that $v$ points towards its interior. Thus $e_{x,v}$ admits a natural parametrization by $S^1_\pi$. Alternatively, $e_{x,v}$ is parametrized as the geodesic $\gamma_{x,-Iv}$, where $I$ is the complex structure on $\C P^n$. 

We interpret $Y_1$ as a space of ``vertical circles" by defining $\varphi:Y_1\to \Lambda \C P^n$ as follows. For each point $y\in e_{x,v}$ there is a unique oriented circle $C_{x,v,y}$ on $\ell_{x,v}$ which passes through $x$ and $y$ and which is tangent to $v$. This circle meets $e_{x,v}$ orthogonally at $x$ and $y$, and it is understood to be constant if $y=x$. The orientation of the circle is determined by requiring that the speed vector at $x$ points towards the interior of $\ell^+_{x,v}$. We then define
$$
\varphi:Y_1\to \Lambda \C P^n,\qquad \varphi(x,v,y):=C_{x,v,y}. 
$$
Let $x^*\in e_{x,v}$ be the antipodal point of $x$ (which corresponds to $\pi/2\in S^1_\pi$). Then $\varphi(x,v,x^*)=C_{x,v,x^*}=\gamma_{x,v}$. We denote 
$$
L_1:=\{(x,v,x^*)\, : \, (x,v)\in ST\C P^n\}\subset Y_1. 
$$
It is then straightforward to see that $(Y_1,L_1,\varphi)$ is a completing manifold for $K_1$. 
Condition~(i) is clearly satisfied since a vertical half-circle $C_{x,v,y}$ is a closed geodesic if and only if $y=x^*$, i.e. $(x,v,y)\in L_1$. Also, the dimension of $Y_1$ equals $\iota(K_1)+\nu(K_1)$. As for condition~(iii), it is satisfied because the inclusion $s:L_1\hookrightarrow Y_1$ is a section of the bundle map $Y_1\to ST\C P^n$. Moreover, the manifold $Y_1$ is orientable. 

\begin{figure}
\begin{center}
\input{CPn.pstex_t}
\end{center}
\caption{Completing manifold for $K_1$ in $\Lambda \C P^n$.}
\label{fig:Y1CPn}
\end{figure}

As in the case of $S^n$, the completing manifold $Y_1$ constitutes the building block for constructing completing manifolds $Y_k$ for all the critical manifolds $K_k$, $k\ge 2$. More precisely, we define 
$$
Y_k:= Y_1 \ {_{\mathrm{ev}}}\!\!\times_{\mathrm{ev}} Y_1 \ {_{\mathrm{ev}}}\!\!\times \dots \times_{\mathrm{ev}} Y_1 \qquad (k\ \mbox{factors})
$$
where $\mathrm{ev}:Y_1\to \C P^n$ is the evaluation map at the origin composed with the representation map $\varphi:Y_1\to \Lambda \C P^n$. We define 
$$
L_k:=\{\big((x,v,x^*), (x,v,x^*),\dots, (x,v,x^*)\big)\, : \, (x,v)\in STS^n\}.
$$
Finally, we define a smooth map $\varphi:Y_k\to \Lambda \C P^n$ by concatenating the $k$ circles with the same origin that are determined by any point in $Y_k$ via the map $\varphi:Y_1\to \Lambda \C P^n$. 

It is again straightforward to see that $(Y_k,L_k,\varphi)$ is a completing manifold for $K_k$. Condition~(i) follows from the fact that the image of a point in $Y_k$ is a geodesic in $K_k$ if and only if it belongs to $L_k$, and also from the fact that $Y_k$ has the correct dimension 
$$
\dim\, Y_k = 2(k+1)n= \iota(K_k)+\nu(K_k).
$$
Finally, that $Y_k$ is orientable follows from the fact that it is a fiber product of orientable manifolds over orientable manifolds.

We obtain therefore an isomorphism with arbitrary coefficients 
$$
H_\cdot(\Lambda \C P^n)\simeq H_\cdot(\C P^n)\oplus \bigoplus_{k\ge 1} H_\cdot(ST\C P^n)[-2(k-1)n-1].
$$
Note that for $n=1$ this computation agrees with the one of $H_\cdot(\Lambda S^2)$. 

Again, a transparent way of representing this computation is in a table in which the horizontal coordinate records the energy level and in which the vertical coordinate records the homological degree. 

\begin{center}
\begin{tabular}
[c]{c|cccccc}
$\vdots$ &  &  &  &  & & \\
$6n+1$ &  &  &  &  & & \\
$6n$ &  &  &  &  &  & \\
$\vdots$ &  &  &  &  & & \\
$4n+1$ &  &  &  &  & \dots & \\
$4n$ &  &  &  &  & & \\
$\vdots$ &  &  &  &  & & \\
$2n+1$ &  &  &  &  & & \\
$2n$ &  &  &  &  & & \\
$\vdots$ &  &  &  &  & & \\
$1$ &  &  &  &  & & \\
$0$ &
\raisebox{0pt}[0pt][0pt]{\raisebox{40pt}{\dbox{$\xymatrix@R=8pt{ \\ H_\cdot(\C P^n) \\ \quad }$}}} &
\raisebox{0pt}[0pt][0pt]{\raisebox{85pt}{\dbox{$\xymatrix@R=24pt{ \\ H_\cdot(ST\C P^n) \\ \quad }$}}} &
\raisebox{0pt}[0pt][0pt]{\raisebox{125pt}{\dbox{$\xymatrix@R=24pt{ \\ H_\cdot(ST\C P^n) \\ \quad \\ }$}}} &
&
\\\hline
& $0$ & $\pi^2$ & $(2\pi)^2$ & & $(3\pi)^2$ & \dots
\end{tabular}
\end{center}

\begin{remark}
The reader is invited to compare this presentation of the homology of free loop spaces of spheres and projective spaces with the more algebraic one in~\cite{CJY}. 
The algebraic presentation in~\cite{CJY} 
contains more information since it also describes the algebra structure with respect to the Chas-Sullivan product. 
However, this algebra structure can also be computed from the geometric presentation that we gave here. This is a good exercise for the keen reader, and has been implemented for the case of a path space in~\cite{Hingston-Oancea}.
\end{remark}

\section{Relationship to symplectic geometry} \label{sec:symplectic}

This section is designed as a brief and necessarily incomplete review of connections between the material discussed in this paper and some questions in Hamiltonian dynamics. 

\subsection{Finsler metrics and Katok's examples} \label{sec:Katok}

There are many similarities between the problem of the existence of closed geodesics and the problem of the existence of periodic orbits for Hamiltonian systems. As has been already hinted at when we explained the linearized Poincar\'e return map in~\S\ref{sec:GM}, the former problem can be viewed as a particular case of the latter: given a Riemannian manifold $M$, the metric induces a bundle isomorphism between $TM$ and $T^*M$ through which the geodesic flow (on $TM$) is carried onto the Hamiltonian flow of the Hamiltonian $H:T^*M\to\R$ given by $H(p,q)=\frac {|p|^2} 2$ (see for example~\cite{GHL}). This is an instance of Legendre transform, which exhibits Riemannian geodesic flows as particular cases of Hamiltonian flows.
 
More generally, given any Hamiltonian $H:T^*M\to\R_+$ which is strictly convex in the fiber direction, the Legendre transform 
$$
\cL_H:T^*M\to TM=T^{**}M, \qquad (x,p)\mapsto (x,\p_p H(x,p))
$$ 
is a global diffeomorphism. If the Hamiltonian $H$ is also homogeneous of degree $2$, the function 
$$
F=\sqrt{H\circ\cL_H^{-1}}
$$ 
is a norm on the fibers of $TM$ such that the unit spheres are smooth and strictly convex. Such a function $F:TM\to\R_+$ is called a \emph{Finsler metric}. 

A Finsler metric allows one to compute the length of curves by the formula $L(\gamma):=\int F(\dot\gamma(t))\, dt$. However, unlike for Riemannian metrics, the length of a curve is not invariant under reversal of the time-direction. This holds only for so-called \emph{symmetric Finsler metrics}, which satisfy the identity $F(v)=F(-v)$. Any Riemannian metric defines a symmetric Finsler metric, but the latter notion is much more general. A coarse measure of the failure from being symmetric is the \emph{reversibility of a Finsler metric}, defined by $\mathrm{rev}(F):=\sup \{F(-v)/F(v)\, : \, 0\neq v\in TM\}$. We always have $\mathrm{rev}(F)\ge 1$ and the Finsler metric is symmetric iff $\mathrm{rev}(F)=1$.

Much of the theory of geodesics for Riemannian metrics can be carried over to Finsler metrics: there is an energy functional $E(\gamma):=\int F^2(\dot\gamma(t))\, dt$, its critical points are locally minimizing and are called geodesics. One can define the index, the nullity, and develop a corresponding Morse theory, albeit not on the space of $H^1$-curves, where the energy functional of a general Finsler metric is not $C^2$, but rather on finite dimensional approximations of the free loop space. The one notable difference with respect to the rest of this paper is the following: 

\begin{center}
\emph{The energy functional for a non-symmetric Finsler metric is $SO(2)$-invariant, and it is $O(2)$-invariant only if the Finsler metric is symmetric.}
\end{center}

For a non-symmetric Finsler metric, we enlarge the notion of geometrically distinct geodesics by including the case of two geodesics which have the same image but different orientations. General references for Finsler metrics are the books~\cite{Shen,Hryniewicz-Salomao}. The author found the papers~\cite{Ziller-Katok,Rademacher1994,Rademacher2004,Rademacher2007} very useful as well.

With one exception, all the results that we discussed in this paper hold for Finsler metrics, since one uses at most the $SO(2)$-symmetry of the problem, and not the $O(2)$-symmetry (see also Remark~\ref{rmk:SO2-O2}). The exception is the existence result for infinitely many closed geodesics on $S^2$, which we have discussed only at bird's eye view level, and where the $O(2)$-symmetry is hidden inside the dynamical arguments of the proof. 

The following example by Katok~\cite{Katok,Ziller-Katok} shows that, in the absence of $O(2)$-symmetry, a Finsler metric may admit only finitely many closed geodesics. The example is that of a non-symmetric Finsler metric on $S^n$, $n\ge 2$ which admits a finite number of geometrically distinct prime closed geodesics, more precisely $n$ of them if $n$ is even, and $n+1$ if $n$ is odd. These geodesics and their iterates are nondegenerate and elliptic (meaning that the eigenvalues of the linearized Poincar\'e return map all lie on the unit circle). The construction applies actually to all globally symmetric spaces of rank $1$. Ziller's paper~\cite{Ziller-Katok} provides a beautiful discussion of Katok's construction.

\begin{example}[Katok's example~\cite{Katok,Ziller-Katok}] \label{example:Katok} 
Let $n=2$ and endow $S^2$ with the round metric of curvature $1$. Consider the $1$-parameter group of isometries $\varphi_t$, $t\in\R$ given by rotations around a fixed axis and denote $V$ its infinitesimal generator. The lift $\tilde\varphi_t$, $t\in \R$ to $TM$ of this $1$-parameter group commutes with the geodesic flow since the $\varphi_t$ are isometries. 

On the Hamiltonian side, denote $H_0:T^*S^2\to\R$, $H_0(p):=|p|$ with respect to the metric dual to the one on $TS^2$, and denote $H_1(p):=\langle p,V\rangle$. Define
$$
H_\eps:=H_0+\eps H_1, \qquad \eps\in[0,1[.
$$
A direct computation shows that the Hamiltonian $\frac 1 2 H_\eps^2$ defines a Finsler metric $F_\eps$ for all $\eps\in[0,1[$. The Finsler metric $F_0$ is the standard metric, and in particular it is symmetric. However, the metric $F_\eps$ is non-symmetric as soon as $\eps>0$. The Hamiltonian flows of $H_0$ and $H_1$ commute, and it is an easy exercise to prove that, in the geometric situation at hand, for $\eps>0$ irrational, the periodic orbits of $H_\eps$ are the periodic orbits of $H_0$ that are invariant under the flow of $H_1$. Equivalently, for $\eps>0$ irrational the closed geodesics of $F_\eps$ are the closed geodesics of $F_0$ which are invariant under the isometry group $\varphi_t$. Their geometric image is thus a fixed equator of $S^2$, and there are exactly two geometrically distinct prime closed geodesics (of lengths $2\pi/(1+\eps)$, respectively $2\pi/(1-\eps)$).
\end{example}

\subsection{Hamiltonian- and Reeb flows as generalizations of geodesic flows} 
 
Many notions concerning geodesics on the one hand and concerning Hamiltonian dynamics on the other hand have been developed in parallel.  As an example, Long~\cite{Long} proved iteration formulae for the Maslov or Conley-Zehnder index which generalize the ones that Bott proved for the index of closed geodesics. Salamon and Zehnder~\cite{SZ92} proved an analogue of Lemma~\ref{lem:1}, with applications to the existence of infinitely many closed periodic orbits of some non-degenerate Hamiltonian systems. Viterbo~\cite{Viterbo} proved resonance formulas for convex Hamiltonian systems in $\R^{2n}$ similar to the ones proved by Rademacher, and these were subsequently generalized by Ginzburg, Kerman, Long, Wang, Hu et al.~\cite{Ginzburg-Kerman,Long-et-al}. 

Techniques in Riemannian geometry and the study of closed geodesics served also as an inspiration in order to formulate or solve problems in Hamiltonian dynamics. The Gromoll-Meyer theorem has a Hamiltonian generalization proved recently by McLean~\cite{McLean2012} and Hryniewicz-Macarini~\cite{Hryniewicz-Macarini}. 
The Conley conjecture, stating the existence of infinitely many periodic orbits for   Hamiltonian flows on the standard symplectic tori $T^{2n}$, can be seen as another Hamiltonian analogue of the problem of the existence of infinitely many closed geodesics. The Conley conjecture was proved by Hingston~\cite{Hingston2009}. Ginzburg~\cite{Ginzburg} proved a more general version, for closed symplectic manifolds $(M,\omega)$ such that $\langle\omega,\pi_2(M)\rangle=0$. Their methods were inspired by previous work on the closed geodesic problem (see also Appendix~\ref{app:S2}). 

The question of the existence of closed orbits of Hamiltonian systems on convex/starshaped energy levels in $\mathbb{R}^{2n}$ has been studied by Ekeland~\cite{Ekeland,Ekeland-book}, Ekeland-Lassoued~\cite{Ekeland-Lassoued}, Ekeland-Lasry~\cite{Ekeland-Lasry}, Viterbo~\cite{Viterbo}, Hofer, Wysocki, and Zehnder~\cite{Hofer-Wysocki-Zehnder}, Long and Zhu~\cite{Long-Zhu}, Wang, Hu, and Long~\cite{Wang-Hu-Long}. With the notable exception of~\cite{Hofer-Wysocki-Zehnder}, the fundamental technique for obtaining these results is to transform the Hamiltonian problem into a variational one using the Legendre-Fenchel transform~\cite{Ekeland-book}. As a matter of fact, some kind of convexity is always needed in order to apply variational techniques to a given problem. In the case of closed geodesics, convexity manifests itself under the disguise of the fact that the Hamiltonian $H(p,q)=\frac {|p|^2} 2$ which generates the geodesic flow is fiberwise convex. 

Reeb flows constitute a vast generalization of geodesic flows. The general definition of a Reeb flow is the following: a (co-oriented) \emph{contact manifold} $(N^{2n-1},\xi)$ is an odd-dimensional manifold endowed with a distribution of hyperplanes $\xi$ that is maximally non-integrable, meaning that $\xi=\ker\, \alpha$ for some $1$-form $\alpha$ with $\alpha\wedge (d\alpha)^{\wedge (n-1)}\neq 0$. The $1$-form $\alpha$ is unique up to multiplication by a positive function, and the previous condition is independent of this choice. The $2$-form $d\alpha$ is non-degenerate on $\xi$ and has an oriented transverse $1$-dimensional kernel generated by a vector field $R_\alpha$ that is uniquely determined by the condition $\alpha(R_\alpha)=1$. This is called the \emph{Reeb vector field associated to $\alpha$}. Its orbits are also called \emph{characteristics}. 
The case in point for our discussion is $N=S_\rho T^*M$, the sphere bundle of radius $\rho>0$ associated to some Riemannian metric on a (closed) manifold $M$. The restriction $\alpha:=pdq|_N$ of the Liouville form is a contact form, and the Hamiltonian vector field of $\frac{|p|^2}2$ is a constant multiple of the Reeb vector field on $S_\rho T^*M$.  

One of the driving conjectures in symplectic topology is the Weinstein conjecture, stating the existence of at least one closed characteristic for any Reeb vector field on any contact manifold (compare with the expected existence of infinitely many closed periodic orbits for geodesic flows!). Although this conjecture motivated some of the most influential papers in the field (\cite{Viterbo-WeinsteinR2n,Hofer-Inventiones,Taubes}), there is still no general agreement concerning the correct multiplicity expectation. This issue should in particular be put into perspective in view of Katok's examples discussed in~\S\ref{sec:Katok}. A promising new conjecture has been recently formulated by Sandon~\cite{Sandon} using the notion of a \emph{translated point}. 

Closed characteristics are currently studied using variants of Floer's construction discussed in Remark~\ref{rmk:Floer} combined with Gromov's theory of pseudoholomorphic curves~\cite{Gromov}. 
Abouzaid's monograph~\cite{Abouzaid-monograph}
gives a thorough account of this method. In a quite different direction, a fascinating analogy between geodesics in Riemannian geometry and pseudo-holomorphic curves in symplectic geometry is described by McDuff~\cite{Dusa-Notices} and Witten~\cite{Witten-Gibbs}. 

\begin{remark}
As another striking historical fact concerning the relationship between closed geodesics and closed characteristics, let us mention that analogues of the homologically invisible closed geodesics discussed in Remark~\ref{rmk:invisible} have appeared independently in contact geometry. These are the so-called ``bad" Reeb orbits, which act as disorienting asymptotes for moduli spaces of pseudoholomorphic curves~\cite{EGH}. 
\end{remark}

\section{Guide to the literature} \label{sec:comments}

The books by Gallot, Hulin and Lafontaine~\cite{GHL} or Chavel~\cite{Chavel-Riemannian_Geometry} are classical texts on Riemannian geometry, which include a treatment of the energy functional. Morse theory is classically explained in Milnor's book~\cite{Milnor-Morse_theory}, with specific applications to the problem of the existence of closed geodesics. Morse-Bott nondegenerate manifolds are first introduced by Bott in~\cite{Bott1954}. General facts about the $H^1$-approach to free loop spaces are explained in Klingenberg's book~\cite{Klingenberg-book}. A beautiful reference concerning closed geodesics is the book by Besse~\cite{Besse}. 

Theorem~\ref{thm:GM} was originally proved by Gromoll and Meyer~\cite{Gromoll-Meyer-JDG1969} with rational coefficients. The original paper is certainly the best reference. The version with $\mathbb{F}_p$-coefficients is stated in~\cite[\S4.2]{Klingenberg-book}, and the proof is the same. Klingenberg claims in~\cite{Klingenberg-book} the existence of infinitely many prime closed geodesics on any closed simply connected Riemannian manifold, but the proof is generally understood to have gaps. Both the generic and the non-generic case depend on the ``divisibility lemma" \cite[\S4.3.4]{Klingenberg-book}, which is ``controversial" according to Hingston~\cite[p.~86]{Hingston1984} (see also~\cite{Bott-indomitable,Bott-oldnew,Ziller-Katok}). 

To the author's knowledge, the best treatment of the iteration of the index formulas is the original paper by Bott~\cite{Bott-iteration}. A good treatment of equivariant cohomology is the paper by Atiyah and Bott~\cite{Atiyah-Bott}, which also discusses the notion of a completing manifold. Ziller's paper~\cite{Ziller} is very well-written, albeit very concise. The same circle of ideas is addressed in a different context in~\cite{Hingston-Oancea}. Operations of Chas-Sullivan type are discussed in relationship with closed geodesics by Goresky and Hingston in~\cite{Goresky-Hingston}. Finsler metrics are nicely discussed by Ziller in~\cite{Ziller-Katok}.

We have barely touched upon Hamiltonian dynamics. Variational methods are described in Ekeland's book~\cite{Ekeland-book}. The book by Hofer and Zehnder~\cite{Hofer-Zehnder} discusses the topic from a modern perspective which avoids pseudoholomorphic curves. A panorama of pseudoholomorphic curves is provided in~\cite{Audin-Lafontaine}, and the full-fledged theory is discussed by McDuff and Salamon in~\cite{McDuff-Salamon}. Floer homology is discussed in detail in the book by Audin and Damian~\cite{Audin-Damian}, while Abouzaid's monograph~\cite{Abouzaid-monograph}
is perhaps the best available reference for a discussion of Floer homology and of the algebraic operations that it carries from the perspective of pseudoholomorphic curves.

And, again, the reader should read Bott's survey~\cite{Bott-oldnew}.

\appendix

\section{The problem of existence of infinitely many closed geodesics on the $2$-sphere\\by Umberto Hryniewicz}\label{app:S2}

\markboth{Umberto Hryniewicz}{Geodesics on the $2$-sphere}

This problem and its solution by Bangert and Franks~\cite{Bangert1993,Franks1992} has roots in the work of Poincar\'e~\cite{Poincare} and Birkhoff~\cite{Birkhoff}. Namely, Birkhoff considers what is now called the ``Birkhoff annulus'' associated to a simple closed geodesic on $S^2$: such a geodesic can be seen as an equator dividing the sphere into two hemispheres, and unit vectors pointing into one of the hemispheres form the associated Birkhoff annulus; it is an embedded open annulus inside the unit tangent bundle consisting of an $S^1$-family of bouquets of unit vectors. Since simple closed geodesics always exist, Birkhoff annuli always exist as well. Birkhoff~\cite{Birkhoff} shows that there is a well-defined first return map of the geodesic flow to this annulus when the curvature is everywhere positive.

\begin{figure}[htp]      
\centering       
\includegraphics[scale=0.18]{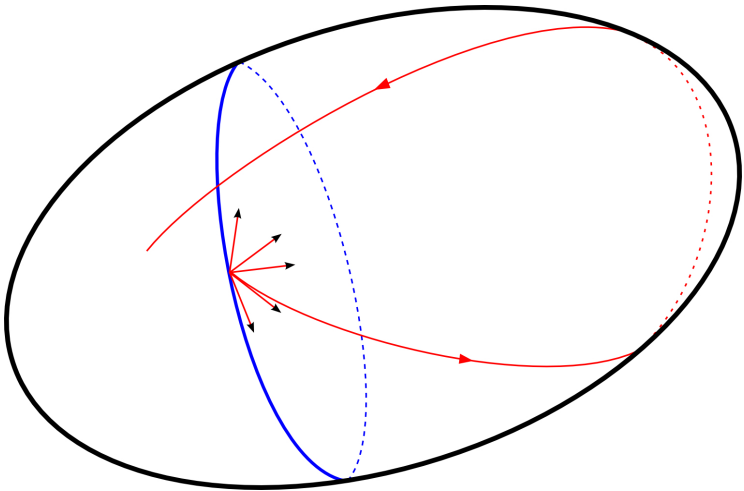}      
\caption{{A Birkhoff annulus in positive curvature: no geodesic ray is trapped to a hemisphere and the return map is well-defined.}}
\end{figure}

This kind of annular global surface of section already appeared in the work of Poincar\'e on the planar circular restricted three-body problem. After Levi-Civita regularization, certain components of the energy levels below the first critical value of the energy are diffeomorphic to $SO(3)$, and results from~\cite{AFKP} tell us that the dynamics on these levels are equivalent to tight Reeb flows. Moreover, for small mass ratio and generic values of the energy there is an annulus-like global surface of section, the return map to this annulus can be extended to the closed annulus and the twist condition for applying the Poincar\'e-Birkhoff theorem is verified.

Returning to geodesic flows on $S^2$, if the return map to a Birkhoff annulus associated to a simple closed geodesic admitting conjugate points (considered as a geodesic ray) exists, then it can be extended via the second conjugate point to a map on a closed annulus. Under the additional assumption that the boundary geodesic has a ``twist in its linearized dynamics'', more precisely that its Poincar\'e inverse rotation number is not equal to 1, the annulus map satisfies the assumptions of the Poincar\'e-Birkhoff theorem, which can be used to deduce $\infty$-many closed geodesics given by fixed points for longer and longer returns.

In general, the return map to a Birkhoff annulus need not be well-defined, but in this case Bangert~\cite{Bangert1993} proves that there are still $\infty$-many closed geodesics. He proceeds in two steps. First, he proves that simple closed geodesics without conjugate points, also referred to as ``waists'' (this analogy is only precise generically in the metric), imply the existence of infinitely many closed geodesics. Second, he shows that if a simple geodesic has conjugate points and the return map to the associated Birkhoff annulus is not well-defined, then one of the hemispheres carries such a waist.

\begin{figure}[htp]      
\centering       
\includegraphics[scale=0.15]{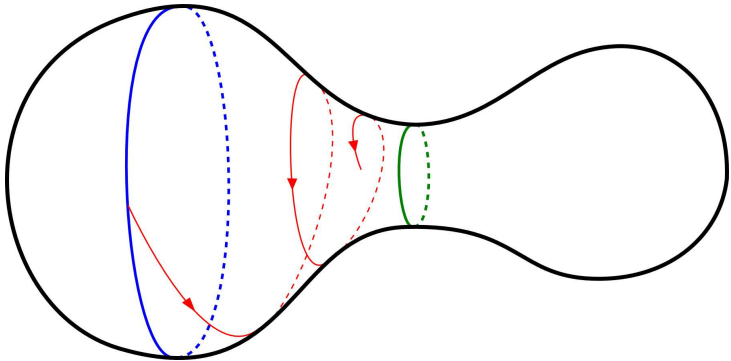}      
\caption{{A Birkhoff annulus in non-positive curvature: there may be waists and, generically, certain geodesic rays are captured by their stable manifolds.}} 
\end{figure}

However, even if the return map to a Birkhoff annulus associated to a simple closed geodesic with conjugate points is defined, the twist condition necessary to apply the Poincar\'e-Birkhoff theorem may not be satisfied (e.g. the round metric). Then, important results of Franks~\cite{Franks1992} provide the final step in the argument: an area-preserving map of the open annulus has $\infty$-many periodic points if it has one periodic point. Note that one can use the Lyusternik-Schnirelmann theorem to obtain a fixed point in this case. Combining these results of Bangert and Franks, the existence of $\infty$-many closed geodesics for Riemannian metrics on $S^2$ is established.

The result of Bangert and Franks can be independently recovered by combining arguments from Hingston~\cite{Hingston1993} and Angenent~\cite{Angenent}. This circle of ideas will be further explored in~\cite{Hingston-Hryniewicz}; here are some of the details.

Following Hingston, if the ``longest'' of the three Lyusternik-Schnirelmann geodes\-ics is nonrotating (Poincar\'e's inverse rotation number is equal to $1$) then it must be a very special degenerate critical point of the energy functional. More precisely, the space of embedded loops in~$S^2$ modulo short ones carries an obvious nontrivial $3$-dimensional homology class -- a certain family of short-long embedded loops -- and Grayson's curve shortening flow can be used to run a min-max argument over this class to obtain a special simple closed geodesic $\gamma_*$. The upshot here is that Grayson's curve shortening flow preserves embeddedness of loops. The index plus the nullity of this geodesic is greater than or equal to $3$, its transverse rotation number is not smaller than $1$, and it has local critical group nontrivial in degree~$3$. When the transverse rotation number of $\gamma_*$ is precisely equal to~$1$, then the sequence given by the index plus nullity of its iterates $\gamma_*^m$ exhibits the slowest possible growth as $m\to\infty$, forcing the very degenerate behavior mentioned above. This class of special critical points introduced by Hingston, and later called \emph{symplectic degenerate maxima/minima (SDM)} by Ginzburg, is crucial to the proofs of the Conley conjecture~\cite{Hingston2009,Ginzburg}. The presence of an SDM forces the existence of $\infty$-many closed geodesics; this is the consequence of the main results from~\cite{Hingston1993,Hingston1997}.

Angenent~\cite{Angenent} now deals with the situation where the transverse rotation number of $\gamma_*$ is greater than $1$; in this case he shows that there are $\infty$-many closed geodesics as well. The proof is based on the curve-shortening flow, and he constructs certain isolating blocks in the sense of Conley. Here, it is as if the twist condition for the return map to the Birkhoff annulus associated $\gamma_*$ is satisfied, except that there may be no return map! He can also say a lot about the kind of geodesics he finds (called $p/q$-satellites).

In particular, the existence of $\infty$-many closed geodesics for Riemannian metrics on $S^2$ is again established as a consequence of a combination of the results of Hingston and Angenent explained above.

Angenent's result can be vastly extended to more general contexts using contact homology, like in~\cite{Hryniewicz-et-al}, allowing for applications to (possibly non-reversible) Finsler geodesic flows on the $2$-sphere and other Hamiltonian systems. The proofs of the results from~\cite{Hryniewicz-et-al} are technically very different from those of~\cite{Angenent}, but the philosophy is the same: we use Floer theoretical arguments to compute the homology of the Conley index of a certain isolating block for the negative gradient flow of the action functional, similarly to what Angenent does for the energy functional.

The above mentioned set of arguments is not entirely contained in the realm of Morse theory, but Morse theory is almost everywhere.

\markboth{Bibliography}{Bibliography}

\bibliographystyle{abbrv}
\bibliography{Geodesics}

\end{document}